\documentclass[british]{scrartcl}
\usepackage[T1]{fontenc}
\usepackage[latin9]{inputenc}
\usepackage{float}
\usepackage{bm}
\usepackage{amsmath}
\usepackage{accents}
\usepackage{amssymb}
\usepackage{graphicx}
\usepackage{esint}
\usepackage[authoryear]{natbib}
\usepackage{amsthm}
\usepackage{url}
\usepackage{subcaption}
\usepackage{enumitem}

\usepackage{todonotes}

\makeatletter
\providecommand\theHALG@line{\thealgorithm.\arabic{ALG@line}}
\makeatother

\makeatletter

\floatstyle{ruled}
\newfloat{algorithm}{tbp}{loa}
\providecommand{\algorithmname}{Algorithm}
\floatname{algorithm}{\protect\algorithmname}

\usepackage{dsfont}\usepackage{algpseudocode}
\usepackage{mathtools}
\usepackage{color}

\newcounter{rmq}[section]
\newcommand{\setX}{\mathcal{X}}

\newcommand{\ui}{[0,1)} 

\renewcommand{\P}{\mathbb{P}}

\newcommand{\F}{\mathcal{F}} 

\newcommand{\Unif}{\mathcal{U}} 

\newcommand{\bx}{x}
\newcommand{\bu}{u}

\newcommand{\by}{y}

\newcommand{\bbx}{X}
\newcommand{\bby}{Y}
\newcommand{\bbu}{U}

\newcommand{\ind}{\mathds{1}}
\newcommand{\dd}{\mathrm{d}}

\newcommand{\bigO}{\mathcal{O}} 
\newcommand{\smallo}{{\scriptscriptstyle\mathcal{O}}} 
\newcommand{\cvz}{\rightarrow 0} 
\renewcommand{\emptyset}{\varnothing} 

\newcommand{\comment}[1]{ \ifthenelse{ \equal{\showcomment}{true} }{ {\bf #1} }{} }
\newcommand{\showcomment}{true}

\newtheorem{thm}{Theorem}

\newtheorem{corollary}{Corollary}
\newtheorem{rem}{Remark}
\newtheorem{lemma}{Lemma}
\newtheorem{definition}{Definition}

\newcommand{\subscript}[2]{$#1 _ #2$}

\makeatother

\usepackage{babel}

\begin{document}

\title{Convergence Results for a Class of Time-Varying Simulated Annealing Algorithms}
\author{Mathieu Gerber\thanks{Present address: University of Bristol, School of Mathematics. Email: mathieu.gerber@bristol.ac.uk (corresponding author)} 
\and Luke Bornn\thanks{Department of Statistics and Actuarial Science, Simon Fraser University}}
\date{}
\maketitle

\begin{abstract}
We provide a  set of conditions which ensure the almost sure convergence of a class of simulated annealing algorithms on a bounded set $\setX\subset\mathbb{R}^d$ based on a time-varying Markov kernel. The class of algorithms considered in this work encompasses the one studied in \citet{Belisle1992} and \citet{Yang2000} as well as its derandomized version recently proposed by \citet{QMC-SA}. To the best of our knowledge, the results we derive 
are the first examples of almost sure convergence results for simulated annealing based on a time-varying kernel. In addition, the assumptions on the Markov kernel and on the cooling schedule have the advantage of being trivial to verify in practice.

\textit{Keywords:} Digital sequences; Global optimization; Simulated annealing
\end{abstract}

\section{Introduction}

Simulated annealing (SA) algorithms are well known tools to evaluate the global optimum of a real-valued function $\varphi$ defined on a measurable set $\setX\subseteq\mathbb{R}^d$. Given a starting value $\bx_0\in\setX$, SA algorithms are determined by a sequence of Markov kernels $(K_n)_{n\geq 1}$, acting from $\setX$ into itself, and a sequence of temperatures (also called cooling schedules) $(T_n)_{n\geq 1}$ in $\mathbb{R}_{>0}$. Simulated annealing algorithms have been extensively studied in the literature and it is now well established that, under mild assumptions on $\varphi$ and on   these two tuning  sequences, the resulting time-inhomogeneous  Markov chain $(\bbx^n)_{n\geq 1}$ is such that the sequence of value functions $(\varphi(\bbx^n))_{n\geq 1}$ converges (in some sense) to $\varphi^*:=\sup_{\bx\in\setX}\varphi(\bx)$. Most of these results are derived  under the condition  $K_n=K$ for all $n\geq 1$, see for instance see \citet{Belisle1992} and \citet{Locatelli2000} for convergence results  of SA on  bounded spaces and \citet{Andrieu2001} and \citet{Rubenthaler2009} for results on unbounded spaces. 
 
However, it is a common practise to use as input of SA a sequence of Markov kernels $(K_n)_{n\geq 1}$ whose variance reduces over time in order to improve local exploration as $n$ increases. For instance, the simulated annealing functions in Matlab (function \texttt{simulannealbnd}) and in R (option ``SANN'' of the function \texttt{optim}) are both based on a Markov kernel whose scale factor is proportional to the current temperature. Some convergence results for such  SA algorithms  based on a time-varying Markov kernel can be found e.g. in \citet{Yang2000}.

Recently, \citet{QMC-SA}  proposed a  modification of SA algorithms where extra dependence among the random variables generated in the course of the algorithm  is introduced to improve the exploration of the state space. The idea behind this new optimization strategy is to replace in SA algorithms the underlying i.i.d.\ uniform random numbers in $\ui$ by points taken from a random sequence with better equidistribution properties. More precisely, \citet{QMC-SA} take for this latter a $(t,s)_R$-sequence, where the parameter $R\in\bar{\mathbb{N}}:=\{0,1,\dots,\infty\}$ controls for the degree of randomness of the input point set, with the case $R=0$ corresponding to i.i.d.\ uniform random numbers and the limiting case $R=\infty$ to a particular construction of quasi-Monte Carlo (QMC) sequences known as $(t,s)$-sequences; see Section \ref{sub:SA_R} for more details. Convergence results and numerical analysis illustrating the good performance of the resulting algorithm are given in \citet{QMC-SA}. Their theoretical analysis only applies for the case where $K_n=K$ for all $n\geq 1$;  in practice, as explained above, it is however desirable to allow the kernels to shrink over time to improve local exploration as the chain becomes more concentrated around the global optimum.

In this work we study  SA type algorithms based on a time-varying kernel by making two important contributions. First, we provide under minimal assumptions an almost sure convergence result for Monte Carlo SA which constitute, to the best of our knowledge, the first almost sure  convergence result for this class of algorithms. Second, we extend the analysis of  \citet{QMC-SA} to the time-varying set-up. As in \citet{Ingber1989} and \citet{Yang2000}, the conditions on the sequence $(K_n)_{n\geq 1}$ for our results to hold amount to imposing a bound on the rate at which the tails of $K_n$ decrease as $n\rightarrow \infty$. Concerning the cooling schedules,  all the results presented in this paper only require that, as in \citet{QMC-SA}, the sequence $(T_n)_{n\geq 1}$ is such that the series $\sum_{n=1}^{\infty}T_n\log(n)$ converges.

The results presented below concern the limit of the sequence  $(\varphi(X^n))_{n\geq 1}$ but, in practise, we are mostly interested in the sequence $\big(\max_{1\leq k\leq n}\varphi(X^k)\big)_{n\geq 1}$ to estimate  $\varphi^*$. However, if $\varphi^*<+\infty$ (as assumed below) it is clear from the relation 
$$
\varphi(X^n)\leq \max_{1\leq k\leq n}\varphi(X^k)\leq\varphi^*,\quad\forall n\geq 1
$$
that the convergence of the former sequence to $\varphi^*$ implies the convergence of the latter sequence to that value.
 
The rest of the paper is organized as follows. Section \ref{sec:prel} introduces the notation and the general class of SA algorithms studied in this work. The main results are provided in Section \ref{sec:Main} and  are illustrated for some classical choice of Markov kernels in Section \ref{sec:ill}. All the proofs are collected in Section \ref{sec:proofs}.

\section{Setting \label{sec:prel}}

\subsection{Notation and conventions}

Let $\setX\subseteq\mathbb{R}^d$,  $\mathcal{B}(\setX)$ be the Borel $\sigma$-field on $\setX$ and $\mathcal{P}(\setX)$ be the set of all probability measures on $(\setX, \mathcal{B}(\setX))$. We write $\mathcal{F}(\setX)$ the set of all Borel measurable functions on $\setX$ and, for $\varphi\in \mathcal{F}(\setX)$,   $\varphi^*=\sup_{\bx\in\setX}\varphi(\bx)$. For integers $b\geq a$  we use the shorthand $a:b$ for the set $\{a,\dots,b\}$ and, for a vector  $\bx=(x_1,\dots, x_d)\in\mathbb{R}^d$, $x_{i:j}=(x_i,\dots,x_j)$ where $ i\leq j\in 1:d$. Similarly, for integer $n\geq 1$, we write $\bx^{1:n}$ the set $\{\bx^1,\dots,\bx^n\}$ of $n$ points in $\mathbb{R}^d$. The ball of radius $\delta>0$ around $\tilde{x}\in\setX$ is denoted in what follows by
$$
B_{\delta}(\tilde{x})=\{x\in\setX:\,\|x-\tilde{x}\|_{\infty}\leq\delta\}
$$
where, for any $z\in\mathbb{R}^d$, $\|z\|_\infty=\max_{i\in 1:d}|z_i|$.

Next, for a Markov kernel $K_n$ acting from  $(\setX, \mathcal{B}(\setX))$  to itself and a point $\bx\in\setX$, we write $F_{K_n}(\bx,\cdot):\setX\rightarrow [0,1]^d$ (resp. $F^{-1}_{K}(\bx,\cdot):[0,1]^d\rightarrow\setX$) the Rosenblatt transformation (resp. the pseudo-inverse Rosenblatt transformation) of the probability measure $K(\bx,\dd y)$; see \citet{Rosenblatt1952} for a definition of these two notions. We denote by $K_{n,i}(\bx,y_{1:i-1}, \dd y_i)$, $i\in 1:d$, the distribution of $y_i$ conditional on $y_{1:i-1}$, relative to $K_n(\bx,\dd y)$ (with the convention $K_{n,i}(\bx,y_{1:i-1}, \dd y_i)=K_{n,1}(\bx, \dd y_1)$ when $i=1$) and the corresponding density  function is denoted by $K_{n,i}(y_i|y_{1:i-1},\bx)$ (again with the convention $K_{n,i}(y_i|y_{1:i-1},\bx)=K_{n,1}(y_1|\bx)$ when $i=1$).

Lastly, we use the shorthand  $\Omega=\ui^{\mathbb{N}}$ and   $\P=\lambda_1^{\otimes\mathbb{N}}$, with $\lambda_d$   the Lebesgue measure on $\mathbb{R}^d$, and consider below the probability space $(\Omega,\F,\P)$ where $\F$ denotes the Borel $\sigma$-algebra on $\Omega$. Throughout this work  we use the convention $\mathbb{N}=\{0,1,\dots\}$ and  the notation $\mathbb{N}_{>0}=\mathbb{N}\setminus\{0\}$, $\bar{\mathbb{N}}=\mathbb{N}\cup\{\infty\}$ and $\mathbb{R}_{>0}=(0,+\infty)$.

\subsection{Simulated annealing algorithms\label{sub:SA}}

Let $(K_n)_{n\geq 1}$ be a sequence of Markov kernels acting from $(\setX,\mathcal{B}(\setX))$ to itself and $(T_n)_{n\geq 1}$ be a sequence in $\mathbb{R}_{>0}$. Then, for $\varphi\in\mathcal{F}(\setX)$, let  $\phi_{\varphi,n}:\setX\times \ui^{d+1}\rightarrow\setX$ be the mapping defined, for  $(\bx,\bu)\in \setX\times \ui^{d+1}$, by
\begin{equation}\label{eq:phi}
\phi_{\varphi,n}(\bx,\bu)=
\begin{cases}
\by_n(\bx,\bu_{1:d})&u_{d+1}\leq A_n(\bx,\bu_{1:d})\\
\bx&u_{d+1}> A_n(\bx,\bu_{1:d})
\end{cases}
\end{equation}
where $\by_n(\bx,\bu_{1:d})=F_{K_n}^{-1}(\bx,\bu_{1:d})$ and where
$$
A_n(\bx,\bu_{1:d})=\exp\Big\{\big(\varphi\circ \by_n(\bx,\bu_{1:d})-\varphi(\bx)\big)/T_n\Big\}\wedge 1.
$$
Next, for $n\geq 1$, we recursively define the mapping $\phi_{\varphi,1:n}:\setX\times\ui^{n(d+1)}\rightarrow\setX$ as
\begin{equation}\label{eq:phi_n}
\phi_{\varphi,1:1}\equiv \phi_{\varphi,1},\,\,\, \phi_{\varphi,1:n}(\bx,\bu^{1:n})=\phi_{\varphi,n}\big(\phi_{\varphi,1:(n-1)}(\bx,\bu^{1:(n-1)}),\bu^n\big),\,\, n\geq 2.
\end{equation}

The quantity $\phi_{\varphi,n}(\bx,\bu)$ ``corresponds'' to the $n$-th iteration of a SA algorithm designed to maximize $\varphi$ where, given the current location $\bx^{n-1}=\bx$, a candidate  value $\by^n=\by_n(\bx,\bu_{1:d})$ is generated using the distribution $K_n(\bx,\dd\by)$ on $\setX$ and is accepted if $u_{d+1}$ is ``small'' compared to $A(\bx,\bu_{1:d})$. Note that the $n$-th value generated by a SA algorithm
with  starting point $\bx^0\in\setX$ and input sequence $(\bu^n)_{n\geq 1}$ in $\ui^{d+1}$ is given by $\bx^n= \phi_{\varphi,1:n}(\bx^0,\bu^{1:n})$.

\subsection{A general class of derandomized SA algorithms\label{sub:SA_R}}

If standard SA algorithms take for input  i.i.d.\ uniform random numbers, the above presentation of this optimization technique outlines the fact that other input sequences can be used. In particular, and as illustrated in \citet{QMC-SA}, the use of $(t,s)_R$-sequences  can lead to dramatic improvements compared to plain Monte Carlo SA algorithms. 

Before introducing $(t,s)_R$-sequences (Definition \ref{def:R} below) we first need to recall the definition of $(t,s)$-sequences \citep[see][Chapter 4, for a detailed presentation of these latter]{dick2010digital}.

For integers $b\geq 2$ and $s\geq 1$, let 
$$
\mathcal{E}_s^b=\Big\{\prod_{j=1}^s\big[a_j b^{-d_j},(a_j+1)b^{-d_j}\big)\subseteq \ui^s,\, a_j,\,d_j\in\mathbb{N},\, a_j< b^{d_j},
\, j\in 1:s\Big\}
$$
be the  set of all $b$-ary boxes (or elementary intervals in base $b$) in $\ui^s$. 

Next, for integers $m\geq 0$ and $0\leq t\leq m$, we say that the set $\{\bu^n\}_{n=0}^{b^m-1}$ of $b^m$ points in $\ui^s$ is  a $(t,m,s)$-net in base $b$ if every $b$-ary box $E\in \mathcal{E}_s^b$ of volume $b^{t-m}$ contains exactly $b^t$ points of the point set $\{\bu^n\}_{n=0}^{b^m-1}$, while the sequence $(\bu^n)_{n\geq 0}$ of points in $\ui^s$ is  called a  $(t,s)$-sequence in base $b$ if, for any integers $a\geq 0$ and $m\geq  t$, the set  $\{\bu^n\}_{n=ab^m}^{(a+1)b^m-1}$ is a $(t,m,s)$-net in base $b$.

\begin{definition}\label{def:R}
Let $b\geq 2$, $t\geq 0$, $s\geq 1$  be integers. Then, we say that  the random sequence $(\bbu_R^n)_{n\geq 0}:\Omega\rightarrow\Omega^s$ of points in $\ui^{s}$  is a $(t,s)_R$-sequence in base $b$, $R\in \bar{\mathbb{N}}$, if, for all $n\geq 0$ and $\omega\in\Omega$ (using the convention that empty sums are null),
$$
\bbu_R^n(\omega)=\big(U_{R,1}^n(\omega),\dots,U_{R,s}^n(\omega)\big),\quad U_{R,i}^n(\omega)=\sum_{k=1}^{R}a_{ki}^nb^{-k}+b^{-R} \omega_{ns+i},\quad i\in 1:s
$$
 where the digits $a_{ki}^n$'s in $0:(b-1)$ are such that  $(u_{\infty}^n)_{n\geq 0}$ is a $(t,s)$-sequence in base $b$. 
\end{definition}
Note that the notation $(u_{\infty}^n)_{n\geq 0}$ has been used instead of $(U_{\infty}^n)_{n\geq 0}$ since this sequence is deterministic.

As already mentioned, when $R=0$ the sequence $(\bbu_R^n)_{n\geq 0}$  reduces to a  sequence of i.i.d.\ uniform random numbers in $\ui^s$. Remark also that the sequence $(\bbu_R^n)_{n\geq 0}$ is such that $\bbu_R^n$ is uniformly distributed into one of the $b^{sR}$ hypercubes that partition $\ui^s$, where the position of that hypercube  depends only on the deterministic part of $\bbu_R^n$.  In addition, for any $R\geq t$, $a\in\mathbb{N}$ and $m\in t:R$, the point set $\{\bbu_R^n\}_{n=ab^m}^{(a+1)b^m-1}$ is a $(t,m,s)$-net in base $b$. 

To simplify the presentation,  Definition \ref{def:R} assumes that  $(t,s)_R$-sequences are constructed from a deterministic $(t,s)$-sequence; that is, it is assumed that the  $a_{ki}^n$'s in $0:(b-1)$ are deterministic. However, all the results presented below also hold for $(t,s)_R$-sequences  build on scrambled $(t,s)$-sequences \citep{Owen1995}. We recall that a scrambled $(t,s)$-sequence in base $b$ is a random sequence $(U^n)_{n\geq 0}$  such that  1) $(U^n)_{n\geq 0}$ is a $(t,s)$-sequence (in base $b$) with probability one and 2) $U^n\sim\Unif(0,1)^s$ for all $n\geq 0$. In that case, it is worth noting that, for any $R\in\mathbb{N}_{>0}$, the $\bbu_R^n$'s are no longer independent and, in particular, the process $(\bbu_R^n)_{n\geq 0}$ is not even Markovian.

The rational for replacing i.i.d.\ uniform random numbers by points taken from a $(t,s)_R$-sequence is explained in detail and illustrated in \citet{QMC-SA}. Here, we recall briefly the two main arguments. First, the  deterministic structure of   $(t,s)_R$-sequences leads to a SA algorithm which is much more robust to the tuning sequences $(K_n)_{n\geq 1}$ and $(T_n)_{n\geq 1}$ than plain Monte Carlo SA. This characteristic is particularly important since it is well known that, for a given objective function $\varphi\in\mathcal{F}(\setX)$ and sequence of kernels $(K_n)_{n\geq 1}$, the performance of SA is very sensitive to the choice of $(T_n)_{n\geq 1}$ \citep[see e.g. the numerical results in][]{QMC-SA}. Second, $(t,s)$-sequences are optimal in term of  dispersion which, informally speaking, means that they efficiently fill the unit hypercube and hence enhance the exploration of the state space \citep[see][Chapter 6, for more details on the notion of dispersion]{Niederreiter1992}.

\section{Consistency of  time-varying SA algorithms\label{sec:Main}}

In this section we provide almost sure  convergence results for the general class of time-varying SA algorithms described in Section \ref{sub:SA} and in Section \ref{sub:SA_R}. In Section \ref{sub:resSA} we separately study  the case $R=0$ (i.e. plain Monte Carlo SA algorithms) which requires the fewest assumptions. Then, we provide in Section \ref{sub:Main_SA}  a result that holds for any $R\in\mathbb{N}$ when $d\geq 1$ and show that, when $d=1$, this latter also holds for the limiting case $R=\infty$.

\subsection{Consistency of adaptive Monte Carlo SA\label{sub:resSA}}

The following result constitutes, to the best of our knowledge, the first almost sure convergence theorem for SA based on a Markov kernel that shrinks over time.

\begin{thm}\label{thm:timConvSA}
Let $\setX\subset\mathbb{R}^d$ be a bounded measurable set and assume that $(K_n)_{n\geq 1}$ satisfies the following conditions
\begin{itemize}
\item for all $n\geq 1$ and $x\in\setX$, $K_n(x,\dd y)=K_n( y| x)\lambda_d(\dd y)$, where $K_n(\cdot |\cdot)$  is continuous on $\setX^2$ and such that $K_n(\by|\bx) \geq \underline{K}_n>0$ for all $(\bx,\by)\in \setX^2$;
\item the sequence $(\underline{K}_n)_{n\geq 1}$ satisfies $\sum_{n=1}^{\infty}\underline{K}_n=\infty$.
\end{itemize} 
Let $\varphi\in\mathcal{F}(\setX)$ be such that there exist  a $\bx^*\in\setX$ satisfying  $\varphi(\bx^*)=\varphi^*$ and  a $\delta_0>0$ such that $\varphi$ is continuous on  $B_{\delta_0}(\bx^*)\subset\setX$. Then, if $\sum_{n=1}^{\infty}T_{n}\log(n)<\infty$, we have, for all $\bx^0\in\setX$,
$$
\lim_{n\rightarrow\infty}\varphi\Big(\phi_{\varphi,1:n}\big(\bx_0,\bbu_0^{1:n}(\omega)\big)\Big)\rightarrow\varphi^*,\quad \P\text{-a.s.}
$$

\end{thm}

\begin{proof}
Let $\varphi\in\mathcal{F}(\setX)$ be as in the statement  of the theorem and $\bx_0\in\setX$ be fixed,  and, for $(\omega,n)\in\Omega\times\mathbb{N}_{>0}$, let
$$
\bbx_0^n(\omega)=\phi_{\varphi,1:n}\big(\bx^0,\bbu_0^{1:n}(\omega)\big),\quad \bby_0^n(\omega)=\by_n\big(\bbx_0^{n-1}(\omega),\bbu_{0,1:d}^n(\omega)\big).
$$

Let $\alpha >0$ so that, by Lemma \ref{lem:R}, $\P$-a.s., $U_{0,d+1}^n(\omega)\geq n^{-(1+\alpha)}$ for all $n$ large enough. Therefore, under the assumptions of the theorem and by \citet[][Lemma 4 and Lemma 5]{QMC-SA}, for $\P$-a.s.,  there exists a  $\bar{\varphi}(\omega)\in\mathbb{R}$ such that
$$
\lim_{n\rightarrow\infty}\varphi(\bbx_0^n(\omega))=\bar{\varphi}(\omega).
$$
 To show that, $\P$-a.s., $\bar{\varphi}(\omega)=\varphi^*$, let $\bx^*$ and $\delta_0>0$ be as in the statement of the theorem and note that, for all $\delta\in(0,\delta_0)$ and for all $n\geq 1$,  
$$
\P\big(\bby_0^n \in B_{\delta}(\bx^*)|Y_0^1 ,\dots, Y_0^{n-1}\big)\geq \underline{K}_n\delta^d.
$$
  To conclude the proof, it remains to show that, for any  $\delta\in(0,\delta_0)$,
\begin{align}\label{eq:int_MC}
\prod_{n=1}^{\infty}(1-\underline{K}_n\delta^d)=0.
\end{align}
Indeed, assuming \eqref{eq:int_MC} is true, for $\P$-almost all $\omega\in\Omega$, the set $B_{\delta}(\bx^*)$ is visited, for any $\delta\in(0,\delta_0)\cap\mathbb{Q}$,  infinitely many times by the sequence $(\bby_0^n(\omega))_{n\geq 1}$ and  therefore the result follows from the continuity of $\varphi$ around $\bx^*$.

To show \eqref{eq:int_MC}, simply note that, using the inequality $\log(1+x)\leq x$ for all $x>-1$ and the continuity of the mapping $x\mapsto \exp(x)$, one has under the assumptions on $(K_n)_{n\geq 1}$,
\begin{align*}
\prod_{n=1}^{\infty}(1-\underline{K}_n\delta^d)=\lim_{N\rightarrow\infty}\exp\bigg\{\sum_{n=1}^{N}
\log(1-\underline{K}_n\delta^d)\bigg\}&=\exp\bigg\{\sum_{n=1}^{\infty}
\log(1-\underline{K}_n\delta^d)\bigg\}\\
&\leq \exp\bigg\{-\delta^d\sum_{n=1}^{\infty}
\underline{K}_n\bigg\}\\
&=0.
\end{align*}
\end{proof}

\begin{rem}
The assumption on  $(T_n)_{n\geq 1}$ comes from \citet[][Theorem 1]{QMC-SA} and is independent from the choice of the  Markov kernels $(K_n)_{n\geq 1}$. We refer the reader to that reference for a discussion on how the condition $\sum_{n=1}^{\infty}T_{n}\log(n)<\infty$ compared to common assumptions on $(T_n)_{n\geq 1}$ that can be found in the literature.
\end{rem}

\begin{rem}
This result is obviously independent of the way we sample from the Markov kernel $K_n(\bx,\dd\by)$ and thus remains valid when we do not use the inverse Rosenblatt transformation approach.
\end{rem}

\begin{rem}
If, for all $n\geq 1$,  $K_n=K$ for a Markov kernel $K$ acting from $(\setX, \mathcal{B}(\setX))$ to itself, then Theorem \ref{thm:timConvSA} reduces to  \citet[][Theorem 3]{QMC-SA}.
\end{rem}

\subsection{Consistency  of derandomized  adaptive SA \label{sub:Main_SA}}

When $R\in\mathbb{N}_{>0}$ the study of the stochastic process generated by the SA algorithm described in Section \ref{sub:SA_R} is more challenging due  to its deterministic underlying  structure. Consequently,  additional assumptions on the objective function and on the sequence $(K_n)_{n\geq 1}$ are needed. However, and as illustrated in Section \ref{sec:ill}, these latter turn out to be, for standard choices of Markov kernels, no stronger  than those needed to establish Theorem \ref{thm:timConvSA}.

\subsubsection{Assumptions and additional notation\label{sub:more_notation}}

For integer $b\geq 2$ and $n\in\mathbb{N}_{>0}$, we write $k_n$ and $r_n$ the integers satisfying
$$
b^{k_n-1}\leq n<b^{k_n},\quad (r_n-1)b^{dR+t}\leq n<r_nb^{dR+t}
$$ 
and we recursively  define the sequence $(k_{R,m})_{m\geq 0}$ in $\mathbb{N}_{>0}$  as follows:
$$
k_{R,0}=1,\quad k_{R,m}=\inf_{n\geq 1}\{b^{k_n}\wedge r_nb^{dR+t}:\,b^{k_n}\wedge r_nb^{dR+t}>k_{R,m-1}\}.
$$
As explained in Section \ref{sub:discussion}, this sequence  is used to determine the frequency  we can adapt the Markov kernel (Assumption \ref{H:thm2_K1} below).


Next, we denote by $\setX_l$, $l\in\mathbb{R}$,  the level sets of $\varphi$; that is
$$
\setX_l=\{x\in\setX:\varphi(x)=l\},\quad l\in\mathbb{R}.
$$
Lastly, we recall  the  definition of the Minkovski content of a set that will be used to impose some smoothness on the objective function.

\begin{definition}
A measurable set $A\subseteq\setX$  has a $(d-1)$-dimensional Minkovski content if
$M(A):=\lim_{\epsilon \downarrow 0}\epsilon^{-1}\lambda_d\big((A)_{\epsilon}\big)<\infty$, 
where, for $\epsilon>0$,  we use the shorthand 
$$
(A)_{\epsilon}:=\{\bx\in\setX:\exists \bx'\in A, \|\bx-\bx'\|_{\infty}\leq\epsilon\}.
$$
\end{definition}

We shall consider the following assumptions on  $\setX$, $(U_R^n)_{n\geq 0}$, $(K_n)_{n\geq 1}$ and $\varphi\in\mathcal{F}(\setX)$.

\begin{enumerate}[label=(\subscript{A}{\arabic*})]
\item\label{H:X} $\setX=[0,1]^d$;
\end{enumerate}

\begin{enumerate}[label=(\subscript{B}{\arabic*})]
\item\label{H:R1} $(\bu_{\infty,1:d}^n)_{n\geq 0}$ is a $(t,d)_R$-sequence;
\item\label{H:R2} $(u_{\infty,d+1}^n)_{n\geq 0}$ is a $(0,1)$-sequence with $u_{\infty,d+1}^0=0$;
\end{enumerate}

\begin{enumerate}[label=(\subscript{C}{\arabic*})]

 \item\label{H:thm2_K1} $K_n=K_{k_{R,m_n}}$ for all $n\in (k_{R,m_n-1}):k_{R,m_n}$ and for a $m_n\in\mathbb{N}_{>0}$;
\item\label{H:K1}  Under \ref{H:X}, for a fixed $\bx\in\setX$, the $i$-th component of  $F_{K_n}(\bx,\by)$ is strictly increasing in $y_i\in[0,1]$, $i\in 1:d$;
\item\label{H:K2} The Markov kernel $K_n(\bx,\dd \by)$ admits a continuous density function $K_n(\cdot|\cdot)$  (with respect to the Lebesgue measure) such that, for a constant $\tilde{K}_n>0$,  
$$
\inf_{(\bx,\by)\in\setX^2}K_{n,i}(y_i|\bx, y_{1:i-1})\geq\tilde{K}_n,\quad\forall i\in 1:d;
$$

\item\label{H:K3} There exists a constant $C_{n}<\infty$ such that, for any $\delta_0>0$ and for all $(\tilde{\bx},\bx')\in \setX^2$  which satisfies  $\lambda_d\big(B_{2\delta_0}(\tilde{x})\cap B_{2\delta_0}(x')\big)=0$,   we have, $\forall \delta\in (0,\delta_0]$ and  $\forall (\bx,\by)\in B_{\delta}(\tilde{\bx})\times  B_{\delta}(\bx')$,
$$
\|F_{K_{n}}(\tilde{\bx}, \bx')-F_{K_{n}}(\bx, \by)\|_{\infty}\leq \delta\,C_{n}.
$$
In addition, there exists a constant $\bar{K}_{n,\delta_0}<\infty$ such that, for all $i\in 1:d$, $
K_{n,i}(y_i|y_{1:i-1},\bx)\leq \bar{K}_{n,\delta_0}$;

\item\label{H:K4} The sequences $(\tilde{K}_n)_{n\geq 1}$, $(C_{n})_{n\geq 1}$ and $(\bar{K}_{n,\delta_0})_{k\geq 1}$,  defined in \ref{H:K2}-\ref{H:K3}, are bounded and such that 
$$
n^{-1/d}/\tilde{K}_n=\bigO(1),\quad  C_{n}/\tilde{K}_{n}=\bigO(1),\quad\bar{K}_{n,\delta_0}=\smallo(1).
$$
\end{enumerate}

\begin{enumerate}[label=(\subscript{D}{\arabic*})]
\item\label{H:thm2_K4} The function $\varphi$ is continuous on $\setX$ and such that 
$$
\sup_{x\in \setX:\varphi(x)<\varphi^*}M(\setX_{\varphi(\bx)})<\infty.
$$
\end{enumerate}

\subsubsection{Discussion of the assumptions\label{sub:discussion}}

Condition \ref{H:X} requires that $\setX=[0,1]^d$ but all the results presented below under  \ref{H:X} also hold when $\setX$ is an arbitrary closed hypercube.

Assumptions \ref{H:R1} and \ref{H:R2} on the input sequence are very weak and are  for instance fulfilled when $(\bu_{\infty}^n)_{n\geq 0}$ is a $(d+1)$-dimensional Sobol' sequence \citep[see, e.g.,][Chapter 8, for a definition]{dick2010digital}.

Assumption \ref{H:thm2_K1} imposes a restriction on the frequency we can adapt the Markov kernel  $K_n$. In particular,  the bigger $R$ is, the less frequently we can change $K_n$. To understand this condition  note that,  for $n$ large enough, we have $k_{R,n}=\tilde{r}_{n}b^{dR+t}$ for some $\tilde{r}_{n}\in\mathbb{N}$. Therefore, for $R\geq t$, 
the point set $\big\{U^i_{R, 1:d}\big\}_{i=k_{R,n}}^{k_{R,n}-1}$  contains exactly $b^t$  points in each of the $b^{dR}$ hypercubes of volume  $b^{-dR}$ that partition $\ui^d$, each of these points being independent and uniformly distributed on the corresponding hypercube of volume  $b^{-dR}$ (see Section \ref{sub:SA_R}). Consequently, and under the other assumptions, the sequence $\big\{F_{K_n}^{-1}(\bx,\bbu_{R,1:d}^i)\big\}_{i=k_{R,n}}^{k_{R,n}-1}$ can reach any region of $\setX$ having positive Lebesgue measure  with strictly positive probability. On the contrary,  if we change the kernel too often this last property may not hold and  the algorithm may fail to converge.

Condition \ref{H:K1}  amounts to assuming that, for any $\bx\in\setX$ and $n\geq 1$, the inverse Rosenblatt transformation $F_{K_n}^{-1}(\bx,\cdot)$ is a well defined function. Given \ref{H:X} and \ref{H:K1}, \ref{H:K2} simply amounts to requiring that, for all $\bx\in\setX$ and $n\geq 1$, the distribution $K_n(\bx,\dd \by)\in\mathcal{P}(\setX)$ is absolutely continuous with respect to the Lebesgue measure and that, for any $\by\in\setX$,  $K_n(y|\cdot )$ is continuous on $\setX$. Next, \ref{H:K3} and \ref{H:K4} impose some conditions on the tail behaviour of $K_n$ as $n\rightarrow \infty$. As illustrated in Section \ref{sec:ill}, \ref{H:K3} and \ref{H:K4} are quite weak and are easily verified for standard choices of Markov kernels.

Lastly,  Assumption \ref{H:thm2_K4} on the objective function $\varphi\in\mathcal{F}(\setX)$ is the same as in \citet{QMC-SA} and is inspired from \citet{He2014}.


\subsubsection{Main results}

The following   theorem   establishes the consistency of SA based on $(t,s)_R$-sequences for any $R\in\mathbb{N}$.

\begin{thm}\label{thm:timeConv}

Assume \ref{H:X}-\ref{H:thm2_K4} and let $(T_n)_{n\geq 1}$ be  such that $\sum_{n=1}^{\infty}T_{n}\log(n)<\infty$. Then, for all $R\in\mathbb{N}$ and for all $\bx^0\in\setX$,
$$
\lim_{n\rightarrow\infty}\varphi\big(\phi_{\varphi,1:n}(\bx_0,\bbu_R^{1:n}(\omega))\big)\rightarrow\varphi^*,\quad\P\text{-a.s.}
$$
\end{thm}

\begin{rem}
The condition  $n^{-1/d}/\tilde{K}_n=\bigO(1)$ in \ref{H:K4} is typically equivalent to the condition $\sum_{n=1}^{\infty}\underline{K}_n=\infty$  given in Theorem \ref{thm:timConvSA} since, typically,  $\underline{K}_n\sim\tilde{K}_n^d$.
\end{rem}

The case $R=\infty$ is more challenging because some odd behaviours are difficult to exclude with  a completely deterministic input sequence. However, we manage to establish a convergence result for  deterministic time-varying SA when the state space is univariate. To this end, we however need to modify \ref{H:K4} and to introduce a new assumption on the sequence $(K_n)_{n\geq 1}$.

\begin{quote}
\begin{enumerate}[label=(\subscript{C'}{\arabic*})]
\setcounter{enumi}{4}
\item\label{H:K4b}  The sequences $(C_{n})_{n\geq 1}$, $(\bar{K}_{n,\delta_0})_{n\geq 1}$ and $(\tilde{K}_n)_{n\geq 1}$,  defined in \ref{H:K2}-\ref{H:K3}, are bounded and such that 
$$
n^{-1/d}/\tilde{K}_n=\smallo(1),\quad  C_{n}/\tilde{K}_{n}=\bigO(1),\quad\bar{K}_{n,\delta_0}=\smallo(1);
$$
\end{enumerate}
\begin{enumerate}[label=(\subscript{C}{\arabic*})]
\setcounter{enumi}{5}
\item \label{H:K5} The sequence $(\bar{K}_{n,\delta_0})_{n\geq 1}$ is such that $n^{-1/d}/\bar{K}_{n,\delta_0}=\smallo(1)$.
\end{enumerate}
\end{quote}

Under this new set of conditions we  prove the following result.

\begin{thm}\label{thm:timeConv_Univ}

Assume $d=1$,  $\sum_{n=1}^{\infty}T_{n}\log(n)<\infty$ and that \ref{H:X}- \ref{H:K3}, \ref{H:K4b}, \ref{H:K5},  \ref{H:thm2_K4}. Then,  for all $x_0\in \setX$, 
$$
\lim_{n\rightarrow\infty}\varphi\big(\phi_{\varphi,1:n}(\bx_0,u_{\infty}^{1:n})\big)\rightarrow\varphi^*.
$$
\end{thm}

\begin{rem}
It is worth noting that the conditions given in Theorem \ref{thm:timeConv_Univ} rule out the case $\tilde{K}_n\sim n^{-1/d}$ and  consequently, in the deterministic version of SA, the tails of the kernel cannot decrease as fast as for the random version (i.e. with $R\in\mathbb{N}$).
\end{rem}

\begin{rem}
When $d=1$, the assumption on the Minkovski content of the level sets given in \ref{H:thm2_K4} amounts to assuming that, for any $l<\varphi^*$, $\setX_l$ is a finite set.
\end{rem}

\begin{rem}
If, for all $n\geq 1$,  $K_n=K$ for a Markov kernel $K$ acting from $(\setX, \mathcal{B}(\setX))$ to itself, then  \ref{H:K3} amounts to assuming that $F_K(\cdot,\cdot)$ is Lipschitz on $\setX^2$. In this set-up, Theorems \ref{thm:timeConv} and \ref{thm:timeConv_Univ} reduce to a  particular case of \citet[][Theorems 1 and 2]{QMC-SA}.
\end{rem}

\section{Examples\label{sec:ill}}

The goal of this section is to show that the assumptions on the sequence of Markov kernels  required by Theorems \ref{thm:timConvSA}-\ref{thm:timeConv_Univ} translate, for standard choices of sequence $(K_n)_{n\geq 1}$, into simple conditions on the rate at which the tails decrease as $n\rightarrow\infty$. 

We focus below on Student's t random walk and to the ASA kernel proposed by \citet{Ingber1989}. For this latter and for Cauchy random walks,  we show that the conditions on the scale factors are the same as for the convergence in probability results of \citet{Yang2000}, which were first proposed by \citet{Ingber1989} using a  heuristic argument.

The proofs of the results presented in this section can be found  in Section \ref{subsec:example_proofs}.

\subsection{Student's t random walks}

We recall that the Student's t distribution on $\mathbb{R}$ with location parameter $\xi\in\mathbb{R}$, scale parameter $\sigma\in\mathbb{R}_{>0}$ and $\nu\in\bar{\mathbb{N}}_{>0}$ degree of freedom, denoted by $t_{\nu}(\xi,\sigma^2)$, has the probability density function  (with respect to the Lebesgue measure) given by
$$
f(x;\xi,\nu,\sigma^2)=\frac{\Gamma\big(\frac{\nu+1}{2}\big)}{\sqrt{\nu\pi}\sigma\,\Gamma(\nu/2)}\Big(1+\frac{(x-\xi)^2}{\nu\sigma^2}\Big)^{-\frac{\nu+1}{2}},\quad x\in\mathbb{R}.
$$
In what follows we write
$$
f_{[0,1]}(x;\xi,\nu,\sigma^2)=\frac{f(x;\xi,\nu,\sigma^2)\ind_{[0,1]}(x)}{\int_{[0,1]} f(y;\xi,\nu,\sigma^2)\lambda_1(\dd y)}
$$
the density of the Student's t distribution $t_{\nu}(\xi,\sigma^2)$ truncated on $[0,1]$.

\begin{corollary}\label{cor:studentRW}
For $\bx\in [0,1]^d$ and $n\geq 1$, let 
$$
K_n(\bx,\dd \by)=\otimes_{i=1}^d f_{[0,1]}(y_i;x_i,\nu, \sigma_{n,i}^2)\lambda_1(\dd y_i)
$$ where, for $i\in 1:d$,  $(\sigma_{n,i})_{n\geq 0}$ is a   non-increasing sequence of strictly positive numbers. Let $\sigma_n=\min\{\sigma_{n,i},\,i\in 1:d\}$.  Then,   $(K_n)_{n\geq 1}$ satisfies the assumptions of Theorem \ref{thm:timConvSA}  if
\begin{align}\label{eq:condRate1}
n^{-1/d}\sigma_{n}\big(1+(\nu \sigma^2_{n})^{-1}\big)^{\frac{\nu+1}{2}}=\bigO(1).
\end{align}
Moreover, if $\lim_{n\rightarrow\infty}\sigma_{n,i}=0$ for all $i$ and $\nu=1$ (Cauchy random walk), conditions \ref{H:K1}-\ref{H:K4} hold under \eqref{eq:condRate1} 
 while \ref{H:K1}-\ref{H:K3}, \ref{H:K4b} and \ref{H:K5} hold if $n^{-1/d}/\sigma_{n}=\smallo(1)$.
\end{corollary}
 
\begin{rem}
We conjecture that the second part of the corollary holds for any $\nu\geq 1$. However, establishing this result for $\nu>1$ is much more involved because, in this case,  the c.d.f. of the resulting Student's t distribution does not admit a ``closed'' form expression. 
\end{rem}

Note that, since the tails of the Student's $t$ distribution become  thinner as $\nu$ increases, the conditions in the above result  become more and more complicated to fulfil as $\nu$ increases. For instance, for Gaussian random walks,  \eqref{eq:condRate1} requires that
$$
n^{-1/d}\sigma_n\exp\big\{(2\sigma_n^2)^{-1}\}=\bigO(1)
$$
while, for  Cauchy random walks (i.e. $\nu=1$), we only need that the sequence $(n^{-1/d}/\sigma_n)_{n\geq 1}$ is bounded.

Condition \eqref{eq:condRate1} for the Cauchy proposal is similar to  \citet[][Corollary 3.1]{Yang2000} who, adapting the proof of \citet[][Theorem 1]{Belisle1992}, derives a convergence in probability result for the sequence $(
\varphi(\bbx_0^n))_{n\geq 1}$. See also \citet{Ingber1989} who found similar rates for  Gaussian and Cauchy random walks with a heuristic argument.

\subsection{Adaptive Simulated Annealing (ASA)\label{sub:ASA}}

For Markov kernels of the form $K_n(x,\dd y)=\otimes_{i=1}^d K_{n,i}(y_i|x_i)\lambda_1(\dd y_i)$, the ability  to perform local exploration may be measured by the rate at which the mass of $K_{n,i}(y_i|x_i)$ concentrates around $x_i$ as $n$ increases; that is, by
$$
\bar{K}_{n,i}=\sup_{(x_i,y_i)\in[0,1]^2}K_{n,i}(y_i|x_i).
$$
For Student's t random walks, it is easy to see that $\bar{K}_{n,i}=\bigO(\sigma^{-1}_{n,i})$. Therefore, because  the rate of the decay of the step size $\sigma_{n,i}$  given in Corollary \ref{cor:studentRW} becomes very slow as $d$ increases,  Student's t random walks may fail to perform good local exploration even in moderate dimensional optimization problems.

To overcome this  limit of the Student's t random walks, \citet{Ingber1989} proposes to use the Markov kernel $K_n(\bx,\dd \by)=\otimes_{i=1}^dK_{n,i}(y_i|x_i)\lambda_1(\dd y_i)$ where
\begin{align}\label{eq:ASA1}
K_{n,i}(y_i|x_i)=\frac{\tilde{K}_{n,i}(y_i|x_i)}{\tilde{K}_{n,i}(x_i,[0,1])}
\end{align}
with $\tilde{K}_{n,i}(x_i, \dd y_i)$ a probability distribution on the set
$
[x_i-1,1+x_i]\supseteq [0,1]
$
with density (with respect to Lebesgue measure) defined, for $x_i\in [0,1]$, by
\begin{align}\label{eq:ASA2}
\tilde{K}_{n,i}(y_i|x_i)=\left\{2\Big(|y_i-x_i|+\sigma_{n,i}\Big)\log(1+\sigma_{n,i}^{-1})
\right\}^{-1}\,\,y_i\in[x_i-1,1+x_i]
\end{align}
and where  $(\sigma_{n,i})_{n\geq 1}$, $i\in 1:d$, are non-increasing sequences of strictly positive numbers. Note that, for $u_i\in [0,1]$,
$$
F^{-1}_{K_{n,i}}(x_i,u_i)=x_i+G_{n,i}\left(F_{\tilde{K}_{n,i}}(x_i,0)+u_i\tilde{K}_{n,i}(x_i,[0,1])\right)
$$
where, for $u\in [0,1]$,  \citep[see][]{Ingber1989}
$$
G_{n,i}(u)=\text{sgn}(u-0.5)\sigma_{n,i}\left[(1+\sigma_{n,i}^{-1})^{|2u-1|}-1\right]
$$
and, for $y_i\in[x_i-1,1+x_i]$,
$$
F_{\tilde{K}_{n,i}}(\bx_i,y_i)=\frac{1}{2}+\frac{\text{sgn}(y_i-x_i)}{2}\frac{\log\left(1+\frac{|y_i-x_i|}{\sigma_{n,i}}\right)}{\log\left(1+\frac{1}{\sigma_{n,i}}\right)}.
$$

For this kernel, we obtain the following result.
\begin{corollary}\label{cor:ASA}
For $\bx\in [0,1]^d$ and $n\geq 1$, let $K_n(\bx,\dd \by)=\otimes_{i=1}^dK_{n,i}(y_i|x_i)\lambda_1(\dd y_i)$ with, for $i\in 1:d$,  $K_{n,i}(y_i|x_i)$ defined by \eqref{eq:ASA1}-\eqref{eq:ASA2} and  $(\sigma_{n,i})_{n\geq 0}$   non-increasing sequences of strictly positive numbers that converge to zero as $n\rightarrow\infty$. Let $\sigma_n=\min\{\sigma_{n,i},\,i\in 1:d\}$.  Then, \ref{H:K1}-\ref{H:K4} hold  if
\begin{align}\label{eq:condRateASA}
n^{-1/d}\log(\sigma_n^{-1})=\bigO(1)
\end{align}
 while \ref{H:K1}-\ref{H:K3}, \ref{H:K4b} and \ref{H:K5} hold if 
$$
n^{-1/d}\log(\sigma_n^{-1})=\smallo(1).
$$
Moreover, under \eqref{eq:condRateASA}, the resulting sequence $(K_n)_{n\geq 1}$ satisfies the assumptions of Theorem \ref{thm:timConvSA}.
\end{corollary}

As for  Cauchy random walks, note that the rate for $\sigma_n$ implied by \eqref{eq:condRateASA} is identical to one obtained by \citet[][Corollary 3.4] {Yang2000} for the convergence (in probability) of the sequence $(\varphi(\bbx_0^n))_{n\geq 1}$. See also \citet{Ingber1989} who find the same rate using a heuristic argument.

\section{Proofs and auxiliary results\label{sec:proofs}}

\subsection{Preliminaries\label{p-remarks}}

We first state a technical lemma that plays a key role to provide conditions on the cooling schedules $(T_n)_{n\geq 1}$.

\begin{lemma}\label{lem:R}
Let $(U_R^n)_{n\geq 0}$ be a $(0,1)_R$-sequence in base $b\geq 2$ such that $u_{\infty}^0=0$. Then, for any $\alpha>0$ and $R\in\bar{\mathbb{N}}$, $\P$-almost surely, $U_R^n(\omega)\geq n^{-(1+\alpha)}$ for all $n$ large enough.
\end{lemma}
\begin{proof}
Let $\alpha>0$ be fixed and assume first that $R\in\mathbb{N}$. Then, for any $n\geq 1$ such that $n^{-(1+\alpha)}\leq b^{-R}$,
$$
\P\big(U_R^n(\omega)<n^{-(1+\alpha)}\big)\leq b^{R}n^{-(1+\alpha)}.
$$
Consequently, noting that a sufficient condition to have  $n^{-(1+\alpha)}\leq b^{-R}$ is that $n\geq b^R$, we have
\begin{align*}
\sum_{n=1}^{\infty}\P\big(U_{R}^n(\omega)<n^{-(1+\alpha)}\big)&\leq b^R
+b^R\sum_{n=1}^{\infty}n^{-(1+\alpha)}<\infty
\end{align*}
 and the result follows by Borel-Cantelli lemma. 

If $R=\infty$  note that, as $u_{\infty}^0=0$ and by the properties of $(0,1)$-sequences in base $b$,  $u_{\infty}^n\geq b^{-k_n}$  for all $n\geq 1$,  where we recall that $k_n$ denotes the smallest integer such that $n<b^{k_n}$. Thus, the result follows when $R=\infty$ from the fact that $[0,n^{-(1+\alpha)})\subseteq [0,b^{-k_n})$ for $n$ sufficiently large.
\end{proof}

We now state a preliminary result that will be repeatedly used in the following and which gives some insights on the assumptions on  $(K_n)_{n\geq 1}$ listed in Section \ref{sub:more_notation}

\begin{lemma}\label{lem:dense_TV}
Let  $K_n:[0,1]^d\times\mathcal{B}([0,1]^d)\rightarrow [0,1]$ be a Markov kernel such that conditions \ref{H:K1}-\ref{H:K3} hold. Let $\delta_0 >0$ and $(\tilde{\bx},\bx')\in [0,1]^{2d}$ be such that  $\lambda_d\big(B_{2\delta_0}(\tilde{x})\cap B_{2\delta_0}(x')\big)=0$. Let $\tilde{C}_{n}=0.5\tilde{K}_{n}\big\{1 \wedge \big(0.25\tilde{K}_{n}/C_{n}\big)^{d}\big\}$, $\bar{\delta}_{n,\delta_0}= 1/\tilde{C}_{n}\wedge 0.5\wedge\delta_0$ and $v_{n}:\mathbb{R}_{>0}\rightarrow\mathbb{R}_{>0}$ be defined by
$$
 v_{n}(\delta)=\delta\Big(1\wedge (0.25\tilde{K}_{
n}/
C_{n})^d\Big),\quad\delta\in \mathbb{R}_{>0}.
$$
Then, for all $\delta\in (0, \bar{\delta}_{n,\delta_0}]$, there exist non-empty closed hypercubes 
$$
\underline{W}_{n}(\tilde{\bx},\bx',\delta)\subset\ui^d,\quad \bar{W}_{n}(\tilde{\bx},\bx',\delta)\subset\ui^d,
$$
respectively of side $\underline{S}_{n,\delta}:=\delta\tilde{C}_{n}$ and $\bar{S}_{n,\delta_0,\delta}:=2.5\delta\bar{K}_{n,\delta_0}\vee 1$, such that
\begin{align*}
\underline{W}_{n}(\tilde{\bx},\bx',\delta)\subseteq  K_{n}\big(\bx, B_{v_{n}(\delta)}(\bx')\big)\subseteq \bar{W}_{n}(\tilde{\bx},\bx',\delta),\quad\forall\bx\in  B_{v_{n,}(\delta)}(\tilde{\bx}).
\end{align*}
\end{lemma}

\begin{proof}
The proof of this result follows from similar  computations as in the proof of \citet[][Lemmas 1,  2 and 6]{QMC-SA} and is thus omitted to save space.
\end{proof}

\begin{rem}
As a corollary, note the following. Let $(\tilde{\bx},\bx')\in\setX^2$ and $\delta>0$ be as in Lemma \ref{lem:dense_TV}. Define
\begin{align}\label{eq:kVal}
k_{n,\delta}=\bigg\lceil t+d-\frac{d\log(\delta \tilde{C}_{n}/3)}{\log b}\bigg\rceil\geq t
\end{align}
and let $\{\bu^{i}\}_{i=0}^{b^{k_{n,\delta}}-1}$ be a $(t,k_{n,\delta},d)$ net. Then, under the assumptions of Lemma \ref{lem:dense_TV}, the point set $\{F_{K_n}^{-1}(\bx^{i},\bu^{i})\}_{i=0}^{b^{k_{n,\delta}}-1}$ contains at least $b^t$ points in the set $B_{\delta}(\bx')$ if $\bx^{i}\in B_{v_{n}(\delta)}(\tilde{\bx})$ for all $i\in 0:(b^{k_{n,\delta}}-1)$. 
\end{rem}

\begin{rem}
Conversely, if $\{\bu^{i}\}_{i=0}^{b^{k}-1}$ is a $(t,k,d)$ net in base $b$ for a $k\geq t+d$, then, under the assumptions of Lemma \ref{lem:dense_TV}, the point set $\{F_{K_n}^{-1}(\bx^{i},\bu^{i})\}_{i=0}^{b^{k}-1}$ contains at least $b^t$ points in the set $B_{\delta_{n,k}}(\bx')$ if $\bx^{i}\in B_{v_{n}(\delta_{n,k})}(\tilde{\bx})$ for all $i\in 0:(b^k-1)$, where
\begin{align}\label{eq:deltaVal}
\delta_{n,k}=3b^{\frac{t+d+1-k}{d}}\tilde{C}_{n}^{-1}.
\end{align}
\end{rem}

Before stating the last  preliminary result we introduce some  additional notation. For $k\in\mathbb{N}_{>0}$, we denote by $E(1/k)=\{E(j,1/k)\}_{j=1}^{k^{d}}$  the splitting of $[0,1]^d$ into closed hypercubes of side $1/k$ and by  $\tilde{E}(1/k)=\{\tilde{E}(j,1/k)\}_{j=1}^{k^{d}}$  the partition of $\ui^d$ into hypercubes of side $1/k$. Note that we need these two different sets of hypercubes because $\setX=[0,1]^d$ while the points of a $(t,s)_R$-sequence belong to $[0,1)^s$.

Next, under \ref{H:thm2_K4}, the following result provides a bound on the number of hypercubes belonging to $E(1/k)$ that are needed to cover the level sets of $\varphi$.

\begin{lemma}\label{lemma:bound_J}
Assume  \ref{H:thm2_K4}. Let $l<\varphi^*$ be a real number and, for $p\in\mathbb{N}_{>0}$, let  $\epsilon_p=2^{-p}$, $\delta_p=2^{-p-1}$ and  $P^l_{p}\subseteq E(\delta_p)$ be the smallest coverage of $(\setX_{l})_{\epsilon_p}$ by hypercubes in $E(\delta_p)$; that is, $|P^l_{p}|$ is the smallest integer in $1:\delta_p^{-d}$ such that  $(\setX_{l})_{\epsilon_p}\subseteq \cup_{W\in P^l_{p}} W$. Let $J^l_{p}\subseteq 1:\delta_p^{-d}$ be such that $j\in J^l_{p}$ if and only if $E(j,\delta_p)\in P^l_{p}$. Then,  there exists a $p_1^*\in\mathbb{N}$ such that, for all $p> p_1^*$, we have 
\begin{align}\label{eq:delta}
|J^l_{p}|\leq  \bar{C}\delta_{p}^{-(d-1)}
\end{align}
  where $\bar{C}<\infty$ is independent of $l$ and $p$.
\end{lemma}
\begin{proof}
See \citet{He2014} and the computations in the proof of \citet[][Lemma 7]{QMC-SA}.
\end{proof}


To conclude this preliminary section we proceed with some further remarks and notation.


Under \ref{H:K4}, the sequence $(C_{n}/\tilde{K}_{n})_{n\geq 1}$ is bounded above by a constant $C<\infty$. Thus, the sequence $(\tilde{K}_{n}/C_{n})_{n\geq 1}$ is bounded below by $C^{-1}>0$ and, consequently, there exists a constant $C_{v}\leq 1$ such that $v_{n}(\delta)\geq v(\delta):=C_{v}\delta$ for all $\delta>0$ and  $n\geq 1$, where $v_{n}(\cdot)$ is as in Lemma \ref{lem:dense_TV}.  In addition, under \ref{H:K4},  the sequence $(\tilde{K}_{n})_{n\geq 1}$ is bounded and therefore there exists a $\bar{\delta}>0$ satisfying $\bar{\delta}\wedge\delta_0\leq \bar{\delta}_{n,\delta_0}$ for all $n\geq 1$, where $\bar{\delta}_{n,\delta_0}$  is as in Lemma \ref{lem:dense_TV}.
Next, for the proof of Theorem \ref{thm:timeConv_Univ} it will be useful to note that, under \ref{H:K4b},  $b^{-k/d}/\tilde{K}_{b^k}\cvz$ as $k\rightarrow\infty$ and thus $\delta_{k+1,b^k}\cvz$ as $k\rightarrow\infty$, with $\delta_{k+1,b^k}$ given by \eqref{eq:deltaVal}.

In what follows, we use the shorthand $r^*=dR+t$ and the integers $N^*$ and $m^*$ are   such that $k_{R,m_n}=r_nb^{r^*}$ for all $n>N^*$ and $mb^{r^*}>N^*$ for all $m>m^*$. For $m\in\mathbb{N}$, we use the shorthand $I_m=\{mb^{r^*},\dots,(m+1)b^{r^*}-1\}$.

From henceforth, we fix  $\varphi\in\mathcal{F}(\setX)$ and $x_0\in\setX$, and define, for $(\omega,n)\in\Omega\times \mathbb{N}_{>0}$,
$$
\bbx_R^n(\omega)=\phi_{\varphi,1:n}\big(\bx^0,\bbu_R^{1:n}(\omega)\big),\quad Y_R^n(\omega)=y_n\Big(X_R^{n-1}(\omega), U_{R,1:d}^{n}(\omega)\big)
$$
and $\varphi_R^n(\omega)=\varphi\big(\bbx_R^n(\omega)\big)$.

\subsection{Auxiliary results}

The following two lemmas are the key ingredients to establish Theorem \ref{thm:timeConv}.

\begin{lemma}\label{lemma:A_set}
Assume \ref{H:X}, \ref{H:thm2_K4}, \ref{H:thm2_K1}-\ref{H:K4}. Let  $m\in\mathbb{N}$  and, for  $p\in\mathbb{N}_{>0}$ and $R\in\mathbb{N}$, let
\begin{align*}
\Omega^{p}_{R, m}&= \Big\{\omega\in\Omega:\,\exists n\in I_m,\,\,\,\bbx_R^{n}(\omega)\not\in \cup_{j\in N_{j^\omega_{mb^{r^*}-1}}}E(j,\delta_{p})\Big\}\\
&\cap \Big\{\omega\in\Omega:\,\forall n\in I_m,\,\, \bbx_R^{n}(\omega)\in\big(\setX_{\varphi_R^{mb^{r^*}-1}(\omega)}\big)_{\epsilon_p},\,\, \varphi_R^{mb^{r^*}-1}(\omega)<\varphi^*\Big\}
\end{align*}
where $\delta_{p}$ and $\epsilon_p$ are as in Lemma \ref{lemma:bound_J}, $i\in N_{j^\omega_{mb^{r^*}-1}}\subset 1:\delta_{p}^{-d}$ if and only if there exists a $j\in j^{\omega}_{mb^{r^*}-1}$ such that $E(i,\delta_{p})$ has one edge in common with $E(j,\delta_{p})$. For $n\geq 1$ and $\omega\in\Omega$, $j^{\omega}_{n}\subset 1:\delta_{p}^{-d}$ is such that $\bbx_R^{n}(\omega)\in E(j,\delta_{p})$ for all $j\in j^{\omega}_{n}$. Let $\Omega^{p}_{R,\infty}=\cup_{i\geq 1}\cap_{m\geq i}\Omega^{p}_{R,m}$. Then, for all $R\in\mathbb{N}$, there exists a $p_2^*\in\mathbb{N}$ such that, for  all $p> p_2^*$, $\P(\Omega^{p}_{R, \infty})=0$.
\end{lemma}

\begin{proof}
Let $R\in\mathbb{N}$, $p_2^*\geq p_1^*$, with $p_1^*$  as in Lemma \ref{lemma:bound_J}, and choose $p\geq p_2^*$  so that  $\epsilon_p\in (0,\epsilon_{p_2^*}]$. Let   $a^{(p)}\in\mathbb{N}$ be such that $a^{(p)}b^{r^*}\geq  N^*$. We now bound $\P(\Omega_m^{p})$ for a $m\geq a^{(p)}$.

We first remark that there exists a  non-negative integer $1\leq \kappa<\infty$ that depends only on $d$  such that, for all $j\in\delta^{-p}$, the set $\cup_{i\in N_j}E(i,\delta_{p})$ is included in a closed hypercube $E(k^*_{j},\delta_{p-\kappa})$.  We assume that $p>\kappa$ from henceforth. Note that each hypercube in $E(\delta_{p-\kappa})$ can be written as the union of $2^\kappa$ hypercubes in $\delta_{p-\kappa}$. For hypercube  $E(j,\delta_{p-\kappa})\in E(\delta_{p-\kappa})$ we denote by $S_j\subset 1:\delta^{\kappa-p}$ the set containing $2^\kappa$ elements such that $ E(j,\delta_{p-\kappa})=\cup_{i\in S_j}E(j,\delta_{p})$. 

Next, let   $k_{p_2^*}\in\mathbb{N}$  be such that  that there exists a $\delta^*\in (0,\bar{\delta}\wedge \delta_{k_{p_2^*}}]$ for which $v(\delta_{k_{p_2^*}})= \delta_{p_2^*}$.  Note that this implies that, for any $p\geq p_2^*$,   $v_{n}(2^{p_2^*-p}\delta^*)\geq \delta_{p}$  and $2^{p_2^*-p}\delta^*\in \big(0,2^{p_2^*-p}(\bar{\delta} \wedge \delta_{k_{p_2^*}})\big]$ for all $n\geq 1$. In what follows we choose $\kappa$  so that $\delta_{p-\kappa}\geq  2^{p_2^*-p}(\bar{\delta} \wedge \delta_{k_{p_2^*}})$;  note that    $\kappa$   does not depend on $p$.

For $k\geq 1$ and $j\in 1:\delta_{k}^{-d}$, let $\bar{\bx}_{k}^j$ be the center of $E(j,\delta_{k})$ and define, for $l<\varphi^*$ and $j\in 1:\delta^{-p}$,
\begin{align*}
\bar{W}^l_{(m+1)b^{r^*}}(j,\delta_{p})&=\bigcup_{j'\neq k^*_j,\, j'\in J^l_{p-\kappa}}\bigcup_{i\in S_{j'}} \bar{W}_{(m+1)b^{r^*}}(\bar{\bx}_{p}^{j},\bar{\bx}_{p}^{i},\delta_{p})\\
&=\bigcup_{j'\neq k^*_j,\, j'\in J^l_{p-\kappa}} \bar{W}_{(m+1)b^{r^*}}(\bar{\bx}_{p}^{j},\bar{\bx}_{p}^{j'},\delta_{p-\kappa})
\end{align*}
where, for $n\geq 1$, $\bar{W}_{n}(\cdot,\cdot,\cdot)\subset\ui^d$  is as in    Lemma \ref{lem:dense_TV} and $J^l_{p-\kappa}$ is as in Lemma \ref{lemma:bound_J}. 

Then, under \ref{H:X}, \ref{H:thm2_K1}-\ref{H:K4}, and by Lemma \ref{lem:dense_TV} a necessary condition to have $\omega\in \Omega^{p}_{R, m}$
is that there exists a $n\in I_m$ such that 
$$
\bbu_R^{n}(\omega)\in W^l_{(m+1)b^{r^*}}(j^{\omega}_{mb^{r^*}-1},\delta_{p}).
$$

Let   $k^{(p)}$ be the largest  integer $k\geq t$ such that $(k-t)/d$ is an integer and such that $b^{k}\leq (2.5\delta_{p-\kappa})^{-d}b^t$, and let $\bar{k}^{(p^*_2)}_{m}$ be the largest integer $k$ which verifies $ b^{k}\leq \bar{K}_{(m+1)b^{r^*},\delta_{p_2^*}}^{-d}$ and such that $k/d$ is an integer. Notice that,  under \ref{H:K4}, $\bar{K}_{k,\delta_{p^*_2}}\rightarrow 0$ a $k\rightarrow\infty$, and therefore we have $\bar{k}^{(p^*_2)}_{m}\rightarrow\infty$ as $m\rightarrow\infty$. Let $k_{m}^{(p)}=k^{(p)}+\bar{k}^{(p_2^*)}_{m}$ and note that, since $\delta_{p-\kappa} \leq\delta_{p^*_2}$ (if necessary one can increase $p$), 
$$
\bar{K}_{(m+1)b^{r^*},\delta_{p-\kappa}}\leq \bar{K}_{(m+1)b^{r^*},\delta_{p_2^*}}.
$$
Consequently, together with Lemma \ref{lem:dense_TV}, this shows that, under \ref{H:X}, \ref{H:thm2_K1}-\ref{H:K4},  the volume of the closed hypercube $\bar{W}_{(m+1)b^{r^*}}(\bar{\bx}_{p}^{j},\bar{\bx}_{p}^{j'},\delta_{p-\kappa})$ is bounded by 
$$
\big(2.5\delta_{p-\kappa}\bar{K}_{(m+1)b^{r^*},\delta_{p_2^*}}\big)^d\leq b^{t-k_{m}^{(p)}}.
$$
Hence, for $j\neq j'$,  $\bar{W}_{(m+1)b^{r^*}}(\bar{\bx}_{p}^{j},\bar{\bx}_{p}^{j'},\delta_{p-\kappa})$  is covered by at most $2^d$ hypercubes of $\tilde{E}(b^{t-k_{m}^{(p)}})$ and thus, for all $j\in J^l_{p}$, $\bar{W}^l_{(m+1)b^{r_{d,r}}}(j,\delta_{p})$ is covered by at most $2^d|J^l_{p-\kappa}|$ hypercubes of $\tilde{E}(b^{(t-k_{m}^{(p)})/d})$.

Take $a^{(p)}$  large  enough so that  $k_{m}^{(p)}>t+dR$ for all $m\geq a^{(p)}$. Then, using the same computations as in \citet[][Lemma 7]{QMC-SA}, we have, for $m\geq a^{(p)}$,
$$
\P\Big(\bbu_R^{n}\in \tilde{E}\big(k,b^{(t-k_{m}^{(p)})/d}\big)\Big)\leq b^{t-k_{m}^{(p)}+dR},\quad \forall k\in 1:b^{k_m^{(p)}-t},\quad  \forall n\in I_m
$$
and thus, under \ref{H:thm2_K4}, using Lemma \ref{lemma:bound_J} (recall that $p_2^*\geq p_1^*$) and the definition of $k_{m}^{(p)}$, we obtain that, for all $j\in J^l_{p}$, $m\geq a^{(p)}$, $l<\varphi^*$ and  $n\in I_m$,
$$
\P\Big(\bbu_R^{n}\in W^l_{(m+1)b^{r^*}}(j,\delta_{p})\Big)\leq 2^d|J^l_{p-\kappa}|b^tb^{t-k_m^{(p)}+dR}\leq \bar{K}_{(m+1)b^{r^*},\delta_{p^*_2}}^{d}C^*\delta_{p-\kappa},
$$
with $ C^*=5^d\bar{C}b^{t+2} b^{dR}$ and $\bar{C}<\infty$ as in Lemma \ref{lemma:bound_J}. Thus, for $m\geq a^{(p)}$ and $l<\varphi^*$, and noting that,  the set $j^{\omega}_{mb^{r^*}-1}$ contains at most $2^d$ elements, we deduce that
$$
\P\Big(\omega\in \Omega^{p}_{R,m}\big | \varphi_R^{mb^{r^*}-1}(\omega)=l\Big)\leq 2^db^{r^*} \bar{K}_{(m+1)b^{r^*},\delta_{p^*_2}}^{d} C^*\delta_{p-\kappa}.
$$
Let $\rho\in (0,1)$. Then, because under \ref{H:K4} the sequence ($\bar{K}_{n,\delta_{p^*_2}})_{n\geq 1}$ is bounded, one can take $p_2^*\in\mathbb{N}$ large enough so that, for all integers $p>p^*_2$ and $m\geq a^{(p)}$, 
$$
\P\Big(\omega\in \Omega^{p}_{R,m}| \varphi_R^{mb^{r^*}-1}(\omega)=l\Big)\leq \rho,\quad\forall l<\varphi^*
$$
so that $\P(\Omega^{p}_{R,m})\leq \rho$ for $p>p^*_2$ and $m\geq a^{(p)}$.
 
To conclude the proof, let $j \geq 1$, $p> p^*$ and $\Omega^{p}_{R,m,j}=\cap_{i=0}^{j-1}\Omega^{p}_{R, m+j}$. Then, it is easily verified that  $\P(\Omega^{p}_{R,m,j})\leq \rho^j$ and,  
consequently, for all $p\geq p_2^*$, $\P(\Omega^{p}_{R,\infty})=0$, as required.
\end{proof}

\begin{lemma}\label{lemma:tildeOmega}
Assume \ref{H:X}-\ref{H:thm2_K4}. Let $R\in\mathbb{N}$ and $x^*\in\setX$ be such that $\varphi^*=\varphi(x^*)$. For $p\in\mathbb{N}_{>0}$, let $\delta_p:=2^{-p-1}$ and
$$
S_p=\{j\in 1:\delta_d^{-p}: \lambda_d\big(B_{2\delta_{p}}(\bar{x}_p^j)\cap B_{2\delta_{p}}(x^*)\big)=0\}
$$
with $\bar{x}_p^j\in\setX$  the center of $E(j, \delta_p)$. Then, we define 
\begin{align*}
\tilde{\Omega}^{p}_{R}=\bigcap_{j\in S_p}\Big\{&\omega\in\Omega:  \,\,\exists n\in I_m,\\
&\,\bbu_R^{n}(\omega)\in \underline{W}_{(m+1)b^{r^*}}\big(\bar{x}^j_{p}, x^*, \delta_{p}\big)\text{ for infinitely many }m\in\mathbb{N}\Big\}
\end{align*}
 with $\underline{W}_{n}(\cdot,\cdot,\cdot)$  as in Lemma \ref{lem:dense_TV}.
Then, for all $R\in\mathbb{N}$ there exists a $p_3^*\in\mathbb{N}$ such that, all $p>p_3*$,
$$
\P\big(\tilde{\Omega}^{p}_{R}\big)=1.
$$

\end{lemma}

\begin{proof}

Let $R\in\mathbb{N}$, $m>m^*$, and $p_3^*$ be such that, for all $p>p_3^*$ and $m>m^*$,
$$
k_{(m+1)b^{r^*},\delta_p}\geq t+d+dR
$$
where, for $n\in\mathbb{N}$ and $\delta>0$, $k_{n,\delta}$ is defined in \eqref{eq:kVal}; note that such a $p_3^*$ exists since, under \ref{H:K4}, the quantity $\tilde{C}_{n}$ that enters in the definition of $k_{n,\delta}$ is bounded uniformly in $n$, and thus, $k_{n,\delta}$ can be made arbitrary large by reducing $\delta$.

Next, for $m\geq 1$, $p>p_3^*$ and $j\in S_p$ (with $S_p$ as in the statement of the lemma), let
$$
D_{p,m}(j):=\Big\{\omega\in\Omega:\,\,\forall n\in I_m,\,\,\bbu_R^{n}(\omega)\not\in \underline{W}_{(m+1)b^{r^*}}\big(\bar{x}_p^j,\bx^*,\delta_p\big)\Big\}.
$$
Then,  to show the lemma  it is enough to prove that, for any $p>p_3^*$, $i\geq 1$  and  $j\in S_p$,
$$
\prod_{m=i}^{\infty}\P\big(D_{p,m}(j)\big)=0.
$$

To this end, remark first that, using the definition of $r^*$, the point set 
$$
P_{m,r^*}:=\{\bu_{\infty}^{n}\}_{n=mb^{r^*}}^{(m+1)b^{r^*}-1}
$$
is a $(t,r^*,d)$-nets in base $b$ which contains, for all $j\in 1:b^{r^*-t}$,  $b^t\geq 1$ points in $\tilde{E}(j,b^{(t-r^*)/d})$.  Consequently, for all $j\in 1:b^{dR}$, $P_{m,r^*}$ has     $b^tb^{r^*-t-dR}=b^{r^*-dR}\geq 1$ points in $\tilde{E}(j,b^{-R})$ and thus, $\P$-a.s., for all $j\in 1:b^{dR}$ the point set  $\{\bbu_R^{n}(\omega)\}_{n=mb^{r^*}}^{(m+1)b^{r^*}-1}$ contains $b^{r^*-dR}$ points in  $\tilde{E}(j,b^{-R})$.   Recall that, for all $n\in I_{m}$, $\bbu_R^{n}$ is uniformly distributed in $\tilde{E}(j_{n},b^{-R})$. 

Next, easy computations shows that, for any $j\in S_p$,  $\underline{W}_{(m+1)b^{r^*}}(\bar{x}_p^j,x^*,\delta_p)$ contains at least one hypercube of the set $\tilde{E}\big(b^{(t_{p}-k_{(m+1)b^{r^*},\delta_p})/d}\big)$, where $t_{p}\in t:(t+d)$ is such that  $
(k_{(m+1)b^{r^*},\delta_p}-t_{p})/d\in\mathbb{N}$, and that each hypercube of the set $\tilde{E}(b^{-R})$ contains 
$$
b^{k_{(m+1)b^{r^*},\delta_p}-t_{p}-dR}\geq b^{k_{(m+1)b^{r^*},\delta_p}-t-d-dR}\geq 1
$$
hypercubes of the set $\tilde{E}\big(b^{(t_p-k_{(m+1)b^{r^*},\delta_p})/d}\big)$. Consequently, for a   $j'\in 1:b^{k_{(m+1)b^{r^*},\delta_p}-t_p}$, we have
\begin{align*}
\rho_{p,m}&:=\P\Big(\omega\in\Omega:\,\exists n\in I_m,\,\bbu_R^{n}(\omega)\in \underline{W}_{(m+1)b^{r^*}}(\bar{x}_p^j,x^*,\delta_d)\Big)\\
&\geq \P\Big(\omega\in\Omega:\,\exists n\in I_m,\,\bbu_R^{n}(\omega)\in E\big(j',b^{(t_p-k_{(m+1)b^{r^*},\delta_p})/d}\big)\Big)\\
&=1-\tilde{\rho}_{p,m}^{b^{r^*-dR}}
\end{align*}
where $\tilde{\rho}_{p.m}:=1-b^{dR+t_{p}-k_{(m+1)b^{r^*},\delta_p}}<1$. This shows that, for all $p>p_3^*$,  $m>m*$ and $j\in S_p$, $\P(D_{p,m}(j))\leq (1-\rho_{p,m})<1$.

To conclude the proof it remains to show that, for $p>p_3^*$, $\sum_{m=1}^{\infty}\log(1-\rho_{p,m})=-\infty$. To see this, remark first that
$$
\sum_{m=1}^{\infty}\log(1-\rho_{p,m})=\sum_{m=1}^{\infty}b^{r^*-dR}\log \tilde{\rho}_{p,m}=\sum_{m=1}^{\infty}b^{r^*-dR}\log \Big(1-b^{dR+t_{p}-k_{(m+1)b^{r^*},\delta_p}}\Big)
$$
where, under \ref{H:K4} and using \eqref{eq:kVal},  $b^{k_{(m+1)b^{r^*},\delta_p}}=\bigO(\tilde{K}_{(m+1)b^{r^*}}^{-d})$  and thus, under \ref{H:K4}, there exists a constant $0<C_p<\infty$ such that $-b^{-k_{(m+1)b^{r^*},\delta_p}}\leq -C_p\,(m+1)b^{r^*}$ for all $m\in\mathbb{N}$. Consequently, using similar computations as in the proof of Theorem \ref{thm:timConvSA}, we deduce that
\begin{align*}
\sum_{m=1}^{M}\log(1-\rho_{d,m})&\leq -b^{r^*-dR}\sum_{m=1}^{\infty}b^{dR+t_{p}-k_{(m+1)b^{r^*},\delta_p}}\\
&\leq -C_p\,b^{2r^*+t_{p}}\sum_{m=1}^{\infty} (m+1)\\
&=-\infty
\end{align*}
as required.
\end{proof}

\subsection{Proof of Theorem \ref{thm:timeConv}\label{p-thm:timeConv}}

Let $R\in\mathbb{N}$ and  $x^*\in\setX$ be such that $\varphi(x^*)=\varphi^*$; note that such a $x^*$ exists since, under \ref{H:X} and \ref{H:thm2_K4}, $\varphi$ is continuous on the compact set $\setX$.

Next,  under  \ref{H:R1}-\ref{H:R2} and the condition  on $(T_n)_{n\geq 1}$, by \citet[][Lemma  5]{QMC-SA} and by  Lemma \ref{lem:R}, there exists a set $\Omega_1\in\mathcal{B}(\Omega)$  such that $\P(\Omega_1)=1$ and such that, for all $\omega\in\Omega_1$, there exists a $\bar{\varphi}(\omega)\in\mathbb{R}$ satisfying $\lim_{n\rightarrow\infty}\varphi(\bbx_R^n(\omega))=\bar{\varphi}(\omega)$. 

Let $p_2^*\in\mathbb{N}$ and $\Omega^{p}_{R,\infty}$ be as in Lemma \ref{lemma:A_set},  $p_3^*\in\mathbb{N}$ be as in Lemma \ref{lemma:tildeOmega}, $p^*=p_2^*\vee p_3^*$, and define
$$
\Omega_2=\bigcap_{p\in\mathbb{N}:\,p> p^*}\Big(\setX\setminus \Omega^{p}_{R,\infty}\Big),\quad \Omega_3=\bigcap_{p>p^*}\tilde{\Omega}^{p}_{R}(x^*).
$$
Then, because  $\mathbb{N}$ is countable, $\P(\Omega_2)=\P(\Omega_3)=1$ by Lemmas \ref{lemma:A_set} and \ref{lemma:tildeOmega}.

Let $\Omega_1'=  \Omega_2 \cap\Omega_3$,
which is such that $\P(\Omega_1')=1$. Consequently, to establish the result it is enough to show that
$$
\bar{\varphi}(\omega)=\varphi^*,\quad\forall \omega\in\Omega_1'.
$$
To this end, remark first that, under \ref{H:thm2_K4},
\begin{align*}
\forall\omega\in\Omega_1',\quad \forall \gamma>0,\quad  \exists N_{\gamma}(\omega)\in\mathbb{N}:\quad  X_R^n(\omega)\in(\setX_{\bar{\varphi}(\omega)})_{\gamma},\quad \forall n\geq N_{\gamma}(\omega).
\end{align*}
Let $\gamma>0$ be fix. Then, under \ref{H:X} and \ref{H:thm2_K4}, $\varphi$ is continuous on the compact set $\setX$ and thus, for any $\omega\in\Omega_1'$, there exists an integer $p_{\omega,\gamma}\in\mathbb{N}$ such that we have both $\lim_{\gamma\cvz}p_{\omega,\gamma}=\infty$ and
\begin{align}\label{eq:inclusion}
(\setX_{\bar{\varphi}(\omega)})_{\gamma}\subseteq (\setX_{\varphi(x)})_{\epsilon_{p_{\omega,\gamma}}},\quad\forall x\in(\setX_{\bar{\varphi}(\omega)})_{\gamma}
\end{align}
where we recall that, for $p\in\mathbb{N}$, $\epsilon_p=2^{-p}$.

Next, for any $\omega\in \Omega_1'$ such that $\bar{\varphi}(\omega)<\varphi^*$, there exists by Lemma \ref{lemma:A_set} a subsequence $(m_i)_{i\geq 1}$ of $(m)_{m\geq 1}$ such that, for $i$ large enough, either 
\begin{align*}
\forall n\in I_{m_i},\,\quad \bbx_R^{n}(\omega)\in \cup_{j\in N_{j^\omega_{mb^{r^*}-1}}}E(j,\delta_{p_{\omega,\gamma}})\subset E(k^*_{mb^{r^*}-1},\delta_{p_{\omega,\gamma}-\kappa})
\end{align*}
or 
\begin{align}\label{eq:Crti1}
\exists n\in I_{m_i}\text{ such that }\bbx_R^{n}(\omega)\not\in (\setX_{\varphi_R^{m_nb^{r^*}-1}(\omega)}\big)_{\epsilon_{p_{\omega,\gamma}}}
\end{align}
for a $k^*_{mb^{r^*}-1}\in 1:\delta^{-(p_{\omega,\gamma}-\kappa)}$ and where the set $N_{j^\omega_{mb^{r^*}-1}}$ is as in Lemma \ref{lemma:A_set} and  $\kappa$ is as in the proof of this latter. If \eqref{eq:Crti1} happens for infinity many $i\in\mathbb{N}$, then, by \eqref{eq:inclusion}, this would contradict the fact that $\omega\in\Omega_1'$. Therefore, for any  $\omega\in \tilde{\Omega}_2:=\{\omega\in\Omega_2':\varphi(\omega)<\varphi^*\}$ there exists a subsequence $(m_i)_{i\geq 1}$ of $(m)_{m\geq 1}$ such, for a $i^*\in\mathbb{N}$,
\begin{align*}
\forall n\in I_{m_i},\, \bbx_R^{n}(\omega)\in E(k^*_{m_ib^{r^*}-1},\delta_{p_{\omega,\gamma}-\kappa}),\quad\forall i\geq i^*.
\end{align*}
Below we use this result to show, by contradiction, that $\P\big(\tilde{\Omega}_2\big)=0$. Assume from henceforth that $\P\big(\tilde{\Omega}_2\big)>0$. To simplify the notation we do as if $\kappa=0$ in what follows.

Let $\omega\in\tilde{\Omega}_2$ be fix from henceforth. Then, let $\gamma$  be small  enough so that there exists a  $p'_{\omega,\gamma}\in\mathbb{N}$ verifying
$$
p_{\omega,\gamma}\wedge p'_{\omega,\gamma}>p^*,\quad\delta_{p_{\omega,\gamma}}\leq \delta_{p'_{\omega,\gamma}}\leq\bar{\delta},\quad \delta_{p_{\omega,\gamma}}\leq v(\delta_{p'_{\omega,\gamma}})\wedge \delta_{p'_{\omega,\gamma}}
$$
and
$$
B_{2\delta_{p'_{\omega,\gamma}}}(\bar{x}^j_{p_{\omega,\gamma}})\cap B_{2\delta_{p'_{\omega,\gamma}}}(x^*)=\emptyset,\quad\forall  j\in   J^{\bar{\varphi}(\omega)}_{p_{\omega,\gamma}}.
$$
Then,  applying Lemma \ref{lem:dense_TV} with $\delta_0=\delta_{p'_{\omega,\gamma}}$ and $\delta=\delta_{p_{\omega,\gamma}}$ yields, for any $m>m^*$ and $j\in   J^{\bar{\varphi}(\omega)}_{p_{\omega,\gamma}}$,
$$
\underline{W}_{(m+1)b^{r^*}}\big(\bar{x}_{p_{\omega,\gamma}}^{j},\bx^*,\delta_{p_{\omega,\gamma}}\big)\subseteq K_{(m+1)b^{r^*}}\big(x, B_{\delta_{p'_{\omega,\gamma}}}(x^*)\big),\quad\forall x\in E(j,\delta_{p_{\omega,\gamma}})
$$
where, for $l<\varphi^*$, $p\in\mathbb{N}$, $J^{l}_{p}$ is as in Lemma \ref{lemma:bound_J}.  

To conclude the proof it suffices to  consider a  $\gamma$  small enough so that one can choose $p'_{\omega,\gamma}$ such that we have both $(\setX_{\bar{\varphi}(\omega)})_{2\gamma}\cap B_{\delta_{p'_{\omega,\gamma}}}(x^*)=\emptyset$ and $\varphi(x)>\varphi(x')$ for all $(x,x')\in (\setX_{\bar{\varphi}(\omega)})_{2\gamma}\times B_{\delta_{p'_{\omega,\gamma}}}(x^*)$. Note that the condition $(\setX_{\bar{\varphi}(\omega)})_{2\gamma}\cap B_{\delta_{p'_{\omega,\gamma}}}(x^*)=\emptyset$ ensures that
$$
E(j,\delta_{p_{\omega,\gamma}})\cap B_{\delta_{p'_{\omega,\gamma}}}(x^*)=\emptyset,\quad\forall  j\in   J^{\bar{\varphi}(\omega)}_{p_{\omega,\gamma}}.
$$
Then, because $\omega\in \Omega_3$, the above computations show that the set $B_{\delta_{p'_{\omega,\gamma}}}(x^*)$ is visited infinitely many times by the sequence $(Y_R^n(\omega))_{n\geq 1}$, which contradicts the fact that $\lim_{n\rightarrow\infty}\varphi(X_R^N(\omega))=\bar{\varphi}(\omega)$ for a $\bar{\varphi}(\omega)<\varphi^*$. Hence, $\tilde{\Omega}_2$ must be empty and the proof is complete.

\subsection{Proof of Theorem \ref{thm:timeConv_Univ}}\label{p-thm:timeConv_Univ}

The proof of this result is based on the proofs of  Lemma \ref{lemma:A_set} and of Theorem \ref{thm:timeConv}. Consequently,  below we only describe the steps that need to be modified. The notation used below is the same as in the proofs of Lemma \ref{lemma:A_set} and Theorem \ref{thm:timeConv}, and is therefore not recalled in the following.

First, in what follows we do as if $\kappa=0$ to simplify the notation. From the proofs of Lemma \ref{lemma:A_set} and of Theorem \ref{thm:timeConv} it must be clear that this assumption will not modify the structure  of the proof of the theorem.

Let $p^*=p_1^*$, with $p^*=p_1^*$ as in Lemma \ref{lemma:bound_J}, $p\in\mathbb{N}$ be such that $p>p^*$,  $\epsilon_p=2^{-p}$ and $N_{\epsilon_p}\in\mathbb{N}$ be such that $x_{\infty}^n\in(\setX_{\bar{\varphi}})_{\epsilon_p}$ for all $n\geq  N_{\epsilon_p}$.

Next, let $m_p\in\mathbb{N}$ be  such that we have both $b^{m_p}> N_{\epsilon_p}$ and $k_{m_p}^{(p)}\leq m_p$. Note that this is always possible to choose such a $m_p$. Indeed, $k_{m}^{(p)}=k^{(p)}+\bar{k}_m^{(p^*)}$  where $k_m^{(p^*)}$ is the largest integer $k$ for which we have both $b^k\leq\bar{K}^{-d}_{b^{m+1},\delta_{p^*}}$ and $(k/d)\in\mathbb{N}$ (see the proof of Lemma \ref{lemma:A_set} with $p_2^*=p^*$). Under \ref{H:K5}, $b^{-m}\bar{K}^{-d}_{b^{m+1},\delta_{p^*}}\cvz$ as $m\rightarrow\infty$ and thus, for $m_p$ large enough, $k_{m_p}^{(p)}< m_p$. Below, we assume $m_p$ is such that $m_p\rightarrow\infty$ as $p\rightarrow\infty$, which is possible under \ref{H:thm2_K4}.

By Lemma \ref{lemma:bound_J}, $|J^{\bar{\varphi}}_{p}|\leq \bar{C}$ when $d=1$, and consequently, the set $\bar{W}^{\bar{\varphi}}_{b^{m_p+1}}(j,\delta_{p})$  contains at most $2^d\bar{C}^2(b-1) b^{t}$ points of the $(t,k_{m_p}^{(p)},1)$-net $\{\bu_{\infty}^{n'}\}_{n'=b^{m_p}}^{b^{m_p}+b^{k_{m_p}^{(p)}}}$. Hence, if for all $n'\geq N_{\epsilon_p}$  only moves from $(\setX_{\bar{\varphi}})_{\epsilon_p}$ to $(\setX_{\bar{\varphi}})_{\epsilon_p}$ occur, then, by Lemma \ref{lem:dense_TV}, for a $$
\tilde{n}\in b^{m_p}:(b^{m_p}+b^{k_{m_p}^{(p)}}-\eta_{p}-1),
$$
the point set $\{x_{\infty}^{n'}\}_{n'=\tilde{n}}^{\tilde{n}+\eta_{p}}$ is such that $x_{\infty}^{n'}\in  E(k^*,\delta_{p})$ for a $k^*\in J^{\bar{\varphi}}_{p}$ and for all $n'\in \tilde{n}:(\tilde{n}+\eta_{p})$, where  $\eta_{p}\geq \big\lfloor \frac{b^{k_{m_p}^{(p)}}}{2^d\bar{C}^2b^t}\big\rfloor$; note that $\eta_{p}\rightarrow\infty$ as $p\rightarrow\infty$ because $k_{m_p}^{(p)}\rightarrow\infty$ as $p\rightarrow\infty$.

As for the proof of Theorem  \ref{thm:timeConv}, we prove the result by contradiction; that is, we show below that if $\bar{\varphi}\neq\varphi^*$, then the point set $\{x_{\infty}^{n'}\}_{n'=\tilde{n}}^{\tilde{n}+\eta_{p}}$ cannot be such that $x_{\infty}^{n'}\in  E(k^*,\delta_{p})$ for a $k^*\in J^{\bar{\varphi}}_{p}$ and for all $n'\in \tilde{n}:(\tilde{n}+\eta_{p})$.

 To see this, let $k^{(p)}_{0}$ be the largest integer $k$ which verifies $\eta_{p}\geq 2b^{k}$, so that $\{\bu_{\infty}^n\}_{n=\tilde{n}}^{\tilde{n}+\eta_{p}}$ contains at least one $(t,k^{(p)}_{0},1)$-net in base $b$; note that $k^{(p)}_{0}\rightarrow\infty$ as $p\rightarrow\infty$. Let $\bx^*\in\setX$ be a global maximizer of $\varphi$, which exists under \ref{H:X} and \ref{H:thm2_K4}. Then, using Lemma  \ref{lem:dense_TV}, there is at least one $n'\in \tilde{n}:(\tilde{n}+\eta_{p})$ such that $F^{-1}_{K_{b^{m_p+1}}}(\bx_{\infty}^{n'-1}, \bu_{\infty}^{n'})\in B_{\delta_{m_p}^{(p)}}(\bx^*)$, with
\begin{align*}
\delta_{m_p}^{(p)}&=3b^{\frac{t+d+1-k^{(p)}_{0}}{d}}\Big(0.5\tilde{K}_{b^{m_p+1}}\left(1 \wedge (0.25\tilde{K}_{b^{m_p+1}}/C_{b^{m_p+1}})^{d}\right)\Big)^{-1}\\
&+ \delta_{p}\Big((0.25\tilde{K}_{b^{m_p+1}}/
C_{b^{m_p+1}})^d\wedge 1\Big)^{-1}.
\end{align*}

To see that this is indeed the case we need to check that all the requirements of Lemma \ref{lem:dense_TV} are fulfilled; that is we need to check that 
\begin{enumerate}
\item\label{1}  $\delta_{m_p}^{(p)}\geq \delta'$ for a $\delta'>0$  such that $k_{b^{m_p+1},\delta'}=k_{0}^{(p)}$;
\item\label{2} $\delta_{p}\leq v_{b^{m_p+1}}\big(\delta_{m_p}^{(p)}\big)$;
\item\label{3} $\delta_{m_p}^{(p)}\leq \bar{\delta}_{b^{m_p+1},\delta_{p^*}}$.
\end{enumerate}

To check  \ref{1}. note that we can take
$$
\delta'=3b^{\frac{t+d+1-m_d}{d}}\Big(0.5\tilde{K}_{b^{m_d+1}}\left(1 \wedge (0.25\tilde{K}_{b^{m_d+1}}/C_{b^{m_d+1}})^{d}\right)\Big)^{-1}
$$
so that $\delta_{m_p}^{(p)}\geq \delta'$ as required. Condition  \ref{2}. holds as well since
\begin{align*}
v_{b^{m_p+1}}(\delta_{m_p}^{(p)})&=\delta_{m_p}^{(p)}\Big((0.25\tilde{K}_{b^{m_p+1}}/
C_{b^{m_p+1}})^d\wedge 1\Big)\\
&=\delta_{p}+\delta'\Big((0.25\tilde{K}_{b^{m_p+1}}/
C_{b^{m_p+1}})^d\wedge 1\Big)\\
&>\delta_{p}
\end{align*}
while \ref{3}. is true for $p^*$ large enough using the remarks of Section \ref{p-remarks}.

To conclude the proof note that, as $p\rightarrow\infty$, $b^{-k^{(p)}_{0}/d}/\tilde{K}_{b^{m_p+1}}\cvz$. To see this, notice that by the definition of $k^{(p)}_{0}$, we have (since $b\geq 2$)
$$
2b^{k^{(p)}_{0}+1}\geq\eta_{p} +1\geq \frac{b^{k_{m_p}^{(p)}}}{2^d \bar{C}^2b^t}.
$$
Thus,
$$
k^{(p)}_{0}\geq k_{m_p}^{(p)}-C,\quad C:=\frac{\log(2^{d-1}\bar{C}^2b^t)}{\log b}+1
$$
and therefore
$$
b^{-k^{(p)}_{0}/d}/\tilde{K}_{b^{m_p+1}}\leq  b^{\frac{C+1}{d}}b^{-\frac{k_{m_p}^{(p)}+1}{d}}/\tilde{K}_{b^{m_p+1}}\cvz
$$
as $m_p\rightarrow\infty$  under ($C_5'$). Thus, since the sequence $(\tilde{K}_k/C_{k,})_{k\geq 1}$ is bounded above under ($C_5'$), this shows that $\delta_{m_p}^{(p)}\rightarrow 0$ as $p\rightarrow\infty$ and the result follows.

\subsection{Proof of Corollary \ref{cor:studentRW} and proof of Corollary \ref{cor:ASA}\label{subsec:example_proofs}}

Conditions \ref{H:K1}-\ref{H:K3} are trivially verified. Below we only show that \ref{H:K4} holds since, from the computations used to establish \ref{H:K4}, it is trivial to verify that \ref{H:K4b}, \ref{H:K5} and the assumptions of Theorem \ref{thm:timConvSA} on $(K_n)_{n\geq 1}$ are verified. To simplify the notation  we assume in the following that $\sigma_{n,i}=\sigma_n$ for all $i\in 1:d$ and for all $n\geq 1$.

\subsubsection{Proof of Corollary \ref{cor:studentRW}}

For $n\geq 1$ and $i\in 1:d$, we use the shorthand  $K_{n,i}(y_i|x_i)=f_{[0,1]}(y_i;x_i,\nu, \sigma_{n}^2)$ and  $\tilde{K}_{n,i}(y_i|x_i)=f(y_i;x_i,\nu, \sigma_{n}^2)$. For $a<b$ we denote by $P_{\nu}(\xi,\sigma,[a,b])$  the probability that $z_i\in [a,b]$ when $z_i\sim t_{\xi}(\mu,\sigma^2)$. 

Since, for all $(\bx,\by)\in\setX^2$ and  $i\in 1:d$, 
$$
K_{n,i}(y_i|x_i)\geq\tilde{K}_{n}:=c_{\nu}\sigma_{n}^{-1}\Big(1+(\nu \sigma^2_{n})^{-1}\Big)^{-\frac{\nu+1}{2}},\quad c_{\nu}=\frac{\Gamma\big(\frac{\nu+1}{2}\big)}{\Gamma(\nu/2)\sqrt{\nu\pi}}
$$
where $n^{-1/d}/\tilde{K}_n=\bigO(1)$ under the assumptions of the corollary,  the first part of \ref{H:K4} is verified.  

To see that the other parts of \ref{H:K4} hold as well, 
let $(\tilde{\bx},\bx')\in\setX^2$ be such that there exists a $\delta_0>0$ which verifies $\lambda_d\big(B_{2\delta_0}(\tilde{\bx})\cap B_{2\delta_0}(\bx')\big)=0$. Let $\delta\in (0,\delta_0]$ and remark that $|x_i-y_i|\geq  \delta_0$ for all $(\bx,\by)\in B_{\delta}(\tilde{\bx})\times B_{\delta}(\tilde{\bx})$. Let $P_n=\sup_{x\in [0,1]}P_{\nu}(x,\sigma_{n},[0,1])$ and note that $P_n\leq P_{n+1}$ for all $n\geq 1$ because the sequence $(\sigma_{n})_{n\geq 1}$ is non-increasing. Therefore, for all $i\in 1:d$, 
$$
K_{n,i}(y_i|x_i)\leq\bar{K}_{n,\delta_0}:=\frac{c_{\nu}}{P_1\sigma_{n}} \Big(1+\delta_0^2(\nu \sigma^2_{n})^{-1}\Big)^{-\frac{\nu+1}{2}},\quad \forall (\bx,\by)\in B_{\delta}(\tilde{\bx})\times B_{\delta}(\tilde{\bx}).
$$

Notice that, under the assumptions of the corollary   we  have both  $\bar{K}^{\gamma}_{n}\cvz$ and $n^{-1/d}/\bar{K}_{n,\delta_0}=\bigO(1)$ as $n\rightarrow\infty$. Hence, the last part of \ref{H:K4} holds.

To show the second part of \ref{H:K4} is verified, let  $(\tilde{\bx},\bx')\in\setX^2$ and $\delta_0>0$ be as above and note that

\begin{equation}\label{eq:bb_asa}
\begin{split}
|F_{K_{n,i}}(x_i, y_i)-F_{K_{n,i}}(\tilde{x}_i, x'_i)|& \leq \frac{\big|\tilde{K}_{n,i}(x_i, [0,y_i])-\tilde{K}_{n,i}(\tilde{x}_i,  [0,x_i'])\big|\big|}{P_1}\\
&+\frac{\big|\tilde{K}_{n,i}(\tilde{x}_i, [0,1])-\tilde{K}_{n,i}(x_i,  [0,1])\big|}{P_1}\\
&\leq \frac{\big|F_{\tilde{K}_{n,i}}(x_i, y_i)-F_{\tilde{K}_{n,i}}(\tilde{x}_i, x'_i)\big|}{P_1}\\
&+\frac{\big|F_{\tilde{K}_{n,i}}(x_i, 1)-F_{\tilde{K}_{n,i}}(\tilde{x}_i, 1)\big|}{P_1}\\
&+2\frac{\big|F_{\tilde{K}_{n,i}}(x_i, 0)-F_{\tilde{K}_{n,i}}(\tilde{x}_i, 0)\big|}{P_1}.
\end{split}
\end{equation}
Next, note that $\text{sgn}(y_i-x_i)=\text{sgn}(x'_i-\tilde{x}_i)$. Assume first that $y_i-x_i\geq 0$ and, without loss of generality, that $y_i-x_i\geq \tilde{x}_i-x'_i\geq 0$. Then, using the fact that the function $\arctan$ is concave on $[0,\infty)$, we have
\begin{align*}
|F_{\tilde{K}_{n,i}}(x_i, y_i)-F_{\tilde{K}_{n,i}}(\tilde{x}_i, x'_i)|&=\frac{1}{\pi}\Big\{\arctan\Big(\frac{y_i-x_i}{\sigma_{n}}\Big)-\arctan\Big(\frac{x'_i-\tilde{x}_i}{\sigma_{n}}\Big)\Big\}\\
&\leq \frac{(y_i-x'_i)-(x_i-\tilde{x}_i)}{\pi \sigma_{n}}\frac{1}{1+\Big(\frac{x'_i-\tilde{x}_i}{\sigma_{n}}\Big)^2}\\
&\leq\frac{2\delta}{\pi\sigma_n}.
\end{align*}
Assume now that $y_i-x_i\leq 0$ and, without loss of generality,  that $y_i-x_i\leq \tilde{x}_i-x'_i<0$. Then, using the fact that the function $\arctan$ is convex on $(-\infty,0]$, we have
\begin{align*}
|F_{\tilde{K}_{n,i}}(x_i, y_i)-F_{\tilde{K}_{n,i}}(\tilde{x}_i, x'_i)|&=\frac{1}{\pi}\Big\{\arctan\Big(\frac{x'_i-\tilde{x}_i}{\sigma_{n}}\Big)-\arctan\Big(\frac{y_i-x_i}{\sigma_{n}}\Big)\Big\}\\
&=\frac{1}{\pi}\Big\{-\arctan\Big(\frac{y_i-x_i}{\sigma_{n}}\Big)-\Big(-\arctan\Big(\frac{x'_i-\tilde{x}_i}{\sigma_{n}}\Big)\Big)\Big\}\\
&\leq -\frac{(y_i-x'_i)-(x_i-\tilde{x}_i)}{\pi \sigma_{n}}\frac{1}{1+\Big(\frac{x'_i-\tilde{x}_i}{\sigma_{n}}\Big)^2}\\
&\leq\frac{2\delta}{\pi\sigma_n}.
\end{align*}
Similarly, repeating these last computations with $y_i=x'_i=0$ and   with $y_i=x'_i=1$ yields, using \eqref{eq:bb_asa}
$$
|F_{K_{n,i}}(x_i,y_i)-F_{K_{n,i}}(\tilde{x}_i, x_i')|\leq \delta C_n,\quad C_n:= \frac{8}{P_1\pi\sigma_n}
$$
and the result follows from the assumptions on $(\sigma_{n})_{n\geq 1}$.

\subsubsection{Proof of Corollary \ref{cor:ASA}}
Since, for all $(\bx,\by)\in\setX^2$ and  $i\in 1:d$, 
$$
K_{n,i}(y_i|x_i)\geq \tilde{K}_{n}:=\frac{1}{2(1+\sigma_{n})\log(1+\sigma_{n}^{-1})}
$$
where $n^{-1/d}/\tilde{K}_n=\bigO(1)$  under the assumptions of the corollary,  the first part of \ref{H:K4} is verified.

To see the other  parts of \ref{H:K4} hold as well, let $P_{n}=\sup_{x_i\in [0,1]}\tilde{K}_{n,i}(x_i, [0,1])$ and note that $P_{n+1}\geq P_n$ for all $n\geq n'$ and for a $n'\geq 1$ large enough since the sequence $(\sigma_n)_{n\geq 1}$ is non-increasing. Therefore, there exists a constant $P_{\setX}>0$ such that $P_n\geq P_{\setX}$ for all $n\geq 1$. Let $(\tilde{\bx},\bx')\in\setX^2$  be such that there exists a $\delta_0>0$ which verifies $\lambda_d\big(B_{2\delta_0}(\tilde{\bx})\cap B_{2\delta_0}(\bx')\big)=0$.

 Let $\delta\in (0,\delta_0]$ and note that,   for all $(\bx,\by)\in B_{\delta}(\tilde{\bx})\times B_{\delta}(\tilde{\bx})$, $|x_i-y_i|\geq \delta_0$ and thus
$$
K_{n,i}(y_i|x_i)\leq \bar{K}_{n,\delta_0}:=\frac{1}{C_{\setX}}\left\{2\left(\delta_0+\sigma_{n}\right)
\log(1+\sigma_{n}^{-1})\right\}^{-1}.
$$
Therefore, $\bar{K}_{n,\delta_0}=\smallo(1)$ under the assumptions on $(\sigma_n)_{n\geq 1}$. Note also that, under the assumptions of the corollary, $n^{-1/d}/\bar{K}_{n,\delta_0}=\bigO(1)$, showing that the first and the last part of \ref{H:K4} hold.

Finally, to show the second part of \ref{H:K4} is verified, let  $(\tilde{\bx},\bx')\in\setX^2$ and $\delta_0>0$ be as above. Let $\delta\in(0,\delta_0]$ and $(\bx,\by)\in B_{\delta}(\tilde{\bx})\times B_{\delta}(\bx')$.  Note that $\text{sgn}(y_i-x_i)=\text{sgn}(x'_i-\tilde{x}_i)$. Without loss of generality we assume that $\text{sgn}(y_i-x_i)=1$ and that $|y_i-x_i|\geq |\tilde{x}_i-x'_i|$. Then, using the fact that $\log(1+x)\leq x$ for any $x> -1$, we have
\begin{align*}
|F_{\tilde{K}_{n,i}}(x_i, y_i)-F_{\tilde{K}_{n,i}}(\tilde{x}_i, x'_i)|&\leq\frac{\log\left(1+\frac{|y_i-x_i|}{\sigma_{n}}\right)-\log\left(1+\frac{|\tilde{x}_i-x'_i|}{\sigma_{n}}\right)}{2\log\left(1+\sigma_{n}^{-1}\right)}\\
&=\frac{\log\left(\frac{\sigma_n+|y_i-x_i|}{\sigma_{n}+|\tilde{x}_i-x'_i|}\right)}{2\log\left(1+\sigma_{n}^{-1}\right)}\\
&=\frac{\log\left(1+\frac{|y_i-x_i|-|\tilde{x}_i-x'_i|}{\sigma_{n}+|\tilde{x}_i-x'_i|}\right)}{2\log\left(1+\sigma_{n}^{-1}\right)}\\
&\leq \frac{\frac{|y_i-x_i|-|\tilde{x}_i-x'_i|}{\sigma_{n}+|\tilde{x}_i-x'_i|}}{2\log\left(1+\sigma_{n}^{-1}\right)}\\
&\leq \frac{\Big||y_i-x_i|-|x_i'-\tilde{x}_i|\Big|}{2\,\sigma_n\,\log\left(1+\sigma_{n}^{-1}\right)}\\
&\leq \frac{|(y_i-x_i)-(x'_i-\tilde{x}_i)|}{2\,\sigma_n\,\log\left(1+\sigma_{n}^{-1}\right)}\\
&\leq \frac{|y_i-x_i'|+|x_i-\tilde{x}_i|}{2\,\sigma_n\,\log\left(1+\sigma_{n}^{-1}\right)}\\
&\leq \frac{\delta}{\sigma_{n}\log\left(1+\sigma_{n}^{-1}\right)}.
\end{align*}
Similarly, repeating these last computations with $y_i=x'_i=0$ and   with $y_i=x'_i=1$ yields, using \eqref{eq:bb_asa} (with $P_1$ replaced by $P_\setX$)
$$
|F_{K_{n,i}}(x_i,y_i)-F_{K_{n,i}}(\tilde{x}_i, x_i')|\leq  \delta C_n\quad C_n:=\frac{4}{\sigma_{n}\log\left(1+\sigma_{n}^{-1}\right)\, P_{\setX}}
$$
and the result follows from the assumptions on $(\sigma_{n})_{n\geq 1}$.

\section*{Acknowledgement}

The authors acknowledge support from DARPA under Grant No. FA8750-14-2-0117.

\bibliographystyle{apalike}
\bibliography{complete}

\end{document}